\documentclass[12pt]{amsart}
\usepackage{amssymb,amscd}
\usepackage{stmaryrd}
\usepackage{epsf, epsfig}
\usepackage{multirow}
\usepackage{colordvi}
\numberwithin{equation}{section}
 
\newtheorem{thm}[equation]{Theorem}
\newtheorem{prop}[equation]{Proposition}
\newtheorem{lemma}[equation]{Lemma}
\newtheorem{cor}[equation]{Corollary}
\newtheorem{ques}[equation]{Question}

\newtheorem{example}[equation]{Example}
\newtheorem{remark}[equation]{Remark}
\newtheorem{definition}[equation]{Definition}

\newenvironment{defn}{\begin{definition}\rm}{\end{definition}}
\newenvironment{ex}{\begin{example}\rm}{\end{example}}
\newenvironment{rem}{\begin{remark}\rm}{\end{remark}}
\newcommand{\C}{{{\mathbb C}}}

\renewcommand{\P}{{{\mathbb P}}}

\newcommand{\N}{{{\mathbb N}}}
\newcommand{\Z}{{{\mathbb Z}}}
\newcommand{\M}{{{\mathcal M}}}
\newcommand{\Q}{{{\mathcal Q}}}
\newcommand{\QQ}{{{\mathbb Q}}}
\newcommand{\Gr}{{{\rm Gr}}}

\newcommand{\SL}{{{\rm SL}}}
\newcommand{\Sp}{{{\rm Sp}}}
\newcommand{\LG}{{{\rm LG}}}
\newcommand{\LM}{{L{\mathcal M}}}
\newcommand{\LQ}{{L{\mathcal Q}}}
\newcommand{\Fdot}{F_{\bullet}}

\newcommand{\bari}{\bar{\imath}}

\newcommand{\angles}[1]{\langle #1 \rangle}
\newcommand{\angbin}[2]{\binom{\langle #1 \rangle}{#2}}
\newcommand{\tangbin}[2]{\tbinom{\langle #1 \rangle}{#2}}
\newcommand{\contr}{\lrcorner}

\newcommand{\one}{{\begin{minipage}[c]{.15cm}\includegraphics[width=.15cm]{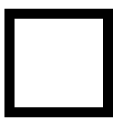}\end{minipage}}}
\newcommand{\strtm}[3]{\P({S^{#1}\C^2}^*)\times{\LM_{#2{-}#1}(#3)}}
\title[Quasimaps for the Lagrangian Grassmannian]
{Quasimaps, straightening laws, and quantum
cohomology for the Lagrangian Grassmannian}

\author{James Ruffo}
\address{Department of Mathematics\\
         SUNY College at Oneonta\\
         Oneonta\\
         NY \ 13820\\
         USA}
\email{ruffojv@oneonta.edu}
\urladdr{http://employees.oneonta.edu/\~{}ruffojv}

\begin{document}
\begin{abstract}
The Drinfel'd Lagrangian Grassmannian
compactifies the space of algebraic maps
of fixed degree from the projective line into
the Lagrangian Grassmannian.  It has a
natural projective embedding arising from
the canonical embedding of the
Lagrangian Grassmannian.  We show that
the defining ideal of any Schubert subvariety
of the Drinfel'd Lagrangian Grassmannian
is generated by polynomials which give a
straightening law on an ordered set.  Consequentially,
any such subvariety is Cohen-Macaulay and Koszul.
The Hilbert function is computed from the straightening
law, leading to a new derivation of certain intersection
numbers in the quantum cohomology ring of the Lagrangian
Grassmannian.
\end{abstract}
%
%
%
%
%
\date{\today}
%
\maketitle

\section{Introduction}

The space of algebraic maps of degree $d$ from $\P^1$
to a projective variety $X$ has applications
to mathematical physics, linear systems theory,
quantum cohomology, geometric
representation theory, and the geometric Langlands
correspondence~\cite{Bra06,So00c,Sot01}.
This space is (almost) never compact, so various
compactifications have been introduced to help
understand its geometry.
Among these (at least when $X$ is a flag variety)
are Kontsevich's space of
{\it stable maps}~\cite{FuPa95,Kon95},
the {\it quot scheme} (or space of {\it quasiflags})
~\cite{Che01,Lau90,Str87}, and
the {\it Drinfel'd compactification}
(or space of {\it quasimaps}).
This latter space is defined concretely as a projective
variety, and much information can be gleaned
directly from its defining equations.

Inspired by the work of Hodge~\cite{Ho43},
Lakshmibai, Musili, Seshadri, and others 
(see~\cite{Lak03,Mus03} and
references therein) developed 
standard monomial theory to study
the {\it flag varieties}
$G/P$, where $G$ is a
semisimple algebraic group
and $P\subseteq G$ is a
parabolic subgroup.  These
spaces have a decomposition
into {\it Schubert cells}, whose
closures (the {\it Schubert varieties})
give a basis for cohomology.
As consequences of standard monomial theory,
Schubert varieties are normal and
Cohen-Macaulay, and one has an explicit
description of their singularities and
defining ideals.

A key part of standard monomial theory
is that any flag variety $G/P$ ($P$ a
parabolic subgroup) has a projective
embedding which presents its coordinate ring
as an {\it algebra with straightening law}
(Definition~\ref{asl.def}),
a special case of a {\it Hodge algebra}~\cite{DEP}.
This idea originates with the
work of Hodge on the Grassmannian~\cite{Ho43},
and was extended to the Lagrangian
Grassmannian by DeConcini and
Lakshmibai~\cite{DeLa79}.

Sottile and Sturmfels have extended standard
monomial theory to the {\it Drinfel'd Grassmannian}
parametrizing algebraic maps from $\P^1$ into
the Grassmannian~\cite{SoSt01}.  They define
Schubert subvarieties of this
space and prove that the
homogeneous coordinate ring of any
Schubert variety (including the Drinfel'd
Grassmannian itself) is an algebra
with straightening law on a distributive lattice.
Using this fact, the authors show that these
Schubert varieties are normal, Cohen-Macaulay and Koszul,
and have rational singularities.

We extend these results to
the {\it Drinfel'd Lagrangian Grassmannian},
which parametrizes
algebraic maps from $\P^1$
into the Lagrangian Grassmannian.
In particular, we prove the following in
Section~\ref{str8law.sec}
(Theorems~\ref{asl.thm} and~\ref{deadon.thm}).

\begin{thm}\label{big.thm}
The coordinate ring of any Schubert subvariety of
the Drinfel'd Lagrangian Grassmannian
is an algebra with straightening law on a doset.
\end{thm}
A doset is a certain kind of ordered
set (Definition~\ref{doset.def}).
As consequences of Theorem~\ref{big.thm},
we show that the
coordinate ring is reduced, Cohen-Macaulay,
and Koszul, and obtain formulas for its degree
and dimension.  These formulas have an interpretation
in terms of quantum cohomology,
as described in Section~\ref{enumgeom.sec}.

In Section~\ref{prelim.sec}, we review the basic
definitions and facts concerning Drinfel'd
compactifications and the Lagrangian
Grassmannian.  Section~\ref{asl.sec} provides
the necessary background on algebras with
straightening law.  We discuss an application
to the quantum cohomology of the Lagrangian
Grassmannian in Section~\ref{enumgeom.sec}.  Our main result
and its consequences are proved
in Section~\ref{str8law.sec}.

We thank Frank Sottile for suggesting this problem,
and for valuable feedback during the preparation of
this paper.

\section{Preliminaries}\label{prelim.sec}

We first give a precise definition of
the Drinfel'd compactification of the space of algebraic
maps from $\P^1$ to a homogeneous variety.
We then review the basic facts we will need regarding
the Lagrangian Grassmannian.

\subsection{Spaces of algebraic maps}
Let $G$ be a semisimple linear algebraic group.
Fix a Borel subgroup $B\subseteq G$
and a maximal torus $T\subseteq B$.
Let $R$ be the set of roots (determined by $T$),
and $S:=\{\rho_1,\dotsc,\rho_r\}$ the
simple roots (determined by $B$).
The simple roots form an ordered basis for
the Lie algebra ${\mathfrak t}$ of $T$;
let $\{\omega_1,\dotsc,\omega_r\}$
be the dual basis (the {\it fundamental weights}).
The {\it Weyl group} $W$ is
the normalizer of $T$ modulo $T$ itself.

Let $P\subseteq G$ be the maximal parabolic
subgroup associated to the fundamental
weight $\omega$, let ${\rm L}({\omega})$
be the irreducible
representation of highest weight $\omega$,
and let $(\bullet,\bullet)$ denote the Killing form on
${\mathfrak t}$.  For $\rho\in R$,
set $\rho^{\vee}:=2\rho/(\rho,\rho)$.
For simplicity, assume that
$(\omega,\rho^{\vee})\leq 2$
for all $\rho\in S$ ({\it i.e.}, $P$
is of {\it classical type}~\cite{Lak03}).
This condition implies that 
${\rm L}(\omega)$ has $T$-fixed lines indexed
by certain admissible pairs of elements of $W/W_P$.

The flag variety $G/P$ embeds in
$\P{\rm L}(\omega)$ as the orbit of
a highest weight line.
Define the {\it degree} of a algebraic map
$f:\P^1\rightarrow G/P$ to be its degree as a
map into $\P{{\rm L}(\omega)}$.

Let $\M_d(G/P)$ be the space
of algebraic maps of
degree $d$ from $\P^1$ into $G/P$.
If $P$ is of classical type
then the set ${\mathcal D}$
of {\it admissible pairs}
indexes homogeneous coordinates on
$\P{\rm L}(\omega)$
(see Definition~\ref{adm.def},
and~\cite{Lak03,Mus03} for a more thorough treatment).
Therefore, any map $f\in \M_d(G/P)$
can be expressed as
$f:[s,t]\mapsto[p_w(s,t)\mid w\in{\mathcal D}]$,
where the $p_w(s,t)$ are homogeneous
forms of degree $d$.
This leads to an embedding of
$\M_d(G/P)$ into $\P((S^d\C^2)^*\otimes {\rm L}(\omega))$,
where $(S^d\C^2)^*$ is the space of homogeneous
forms of degree $d$ in two variables.
The coefficients of the homogeneous forms
in $(S^d\C^2)^*$ give coordinate functions on
$(S^d\C^2)^*\otimes {\rm L}(\omega)$;
they are indexed by the set
$\{w^{(a)}\mid w\in {\mathcal D}, a=0,\dotsc,d\}$,
a disjoint union of $d{+}1$ copies of ${\mathcal D}$.

The closure of
$\M_d(G/P)\subseteq\P((S^d\C^2)^*\otimes {\rm L}(\omega))$ is
called the {\it Drinfel'd compactification} and denoted
$\Q_d(G/P)$.  This definition is due to
V. Drinfel'd, dating from the mid-1980s.
Drinfel'd never published this
definition himself; to the author's
knowledge its first appearance in
print was in~\cite{Ros94}; see also~\cite{Kuz97}.
 
Let $G=\SL_n(\C)$ and $P$
be the maximal parabolic subgroup
stabilizing a fixed $k$-dimensional
subspace of $\C^n$, so that
$G/P=\Gr(k,n)$.  In this case we denote
the {\it Drinfel'd Grassmannian}
$\Q_d(G/P)$ by $\Q_{d}(k,n)$.
In~\cite{SoSt01} it is shown that
the homogeneous coordinate
ring of $\Q_{d}(k,n)$
is an algebra with straightening
law on the distributive lattice
$\binom{[n]}{k}_d:=
\{\alpha^{(a)}\mid \alpha\in\binom{[n]}{k}, 0\leq a\leq d\}$,
with partial order on $\binom{[n]}{k}_d$
defined by $\alpha^{(a)}\leq\beta^{(b)}$
if and only if $a\leq b$ and
$\alpha_i\leq\beta_{b{-}a{+}i}$
for $i=1,\dotsc,k{-}b{+}a$.
It follows that the homogeneous coordinate
ring of $\Q_{d}(k,n)$
is normal, Cohen-Macaulay,
and Koszul, and that the ideal
$I_{k,n{-}k}^d\subseteq\C[\binom{[n]}{k}_d]$ has a
quadratic Gr\"{o}bner basis
consisting of the straightening
relations.

Taking $d=0$ above, one recovers
the classical {\it Bruhat order} on
$\binom{[n]}{k}$.  For a semisimple
algebraic group $G$ with parabolic subgroup $P$, this is an
ordering on the set of maximal coset representatives
of the quotient of the Weyl group of $G$
by the Weyl group of $P$.

Suppose that $d=\ell k{+}q$ for positive
integers $\ell$ and $q$ with $q<k$,
and let $X=(x_{ij})_{1\leq i,j\leq n}$
be a matrix with polynomial entries
$x_{ij}=x_{ij}^{(k_i)}t^{k_i}+\dotsb+x_{ij}^{(1)}t+x_{ij}^{(0)}$,
where $k_i=\ell{+}1$ if $i\leq q$
and $k_i=\ell$ if $i>q$.
The ideal $I_{k,n{-}k}^{d}$ is
the kernel of the map
$\varphi:\C[\binom{[n]}{k}_d]\rightarrow \C[X]$ sending
the variable $p_{\alpha}^{(a)}$ indexed by
$\alpha^{(a)}\in\binom{[n]}{k}_d$
to the coefficient of $t^a$ in the maximal minor
of $X$ whose columns are indexed by $\alpha$.
 
The main results of~\cite{SoSt01}
follow from the next proposition.
Given any distributive lattice,
we denote by $\wedge$ and
$\vee$, respectively, the
{\it meet} and {\it join}.  The
symbol $\wedge$ will also be
used for exterior products
of vectors, but the meaning should
be clear from the context.
\begin{prop}\cite[Theorem 10]{SoSt01}
Let $\alpha,\beta$ be a pair of
incomparable variables in the poset
$\binom{[n]}{k}_d$.  There is
a quadratic polynomial
$S(\alpha,\beta)$ in the kernel of
$\varphi:\C[\binom{[n]}{k}_d]\rightarrow \C[X]$
whose first two monomials are
\[
p_{\alpha}p_{\beta}-
              p_{\alpha\wedge\beta}
              p_{\alpha\vee\beta}\,.
\]
Moreover, if $\lambda p_{\gamma}p_{\delta}$
is any non-initial
monomial in $S(\alpha,\beta)$, then
$\alpha,\beta$ lies in the
interval $[\gamma,\delta]
=\{\theta\in\binom{[n]}{k}_d\mid
\gamma\leq\theta\leq\delta\}$.
\end{prop}
The quadratic polynomials $S(\alpha,\beta)$
in fact form a Gr\"obner basis for the ideal
they generate.  It is shown in~\cite{SoSt01}
that there exists a toric ({\sc sagbi}) deformation taking
$S(\alpha,\beta)$
to its {\it initial form}
$p_{\alpha}p_{\beta}-
             p_{\alpha\wedge\beta}
             p_{\alpha\vee\beta}$,
deforming the Drinfel'd Grassmannian into a toric variety.

Our goal is to extend the main results
of standard monomial theory
to the {\it Lagrangian Drinfel'd Grassmannian}
$\LQ_d(n):=\Q_d(\LG(n))$
of degree-$d$ maps from $\P^1$
into the Lagrangian Grassmannian.

\subsection{The Lagrangian Grassmannian}\label{lag.ssec}

In its natural projective embedding,
the Lagrangian Grassmannian $\LG(n)$
is defined by quadratic relations
which give a straightening law on
a doset~\cite{DeLa79}.
These relations are obtained
by expressing $\LG(n)$ as
a linear section of $\Gr(n,2n)$.
While this is well-known, the
author knows of no explicit
derivation of these relations
which do not require the
representation theory of
semisimple algebraic groups.  We
provide a derivation which does not
rely upon representation theory
(although we adopt the
notation and terminology).
This will
be useful when we consider the
Drinfel'd Lagrangian Grassmannian,
to which representation theory has
yet to be successfully applied.

Set $[n]:=\{1,2,\dotsc,n\}$, $\bari:=-i$,
and $\angles{n}:=
\{\bar{n},\dotsc,\bar1,1,\dotsc n\}$.
If $S$ is any set, let
$\binom{S}{k}$ be the collection of subsets
$\alpha=\{\alpha_1,\dotsc,\alpha_k\}$ of cardinality $k$.

The projective space
$\P(\bigwedge^n\C^{2n})$ has {\it Pl\"ucker coordinates}
indexed by the distributive lattice $\angbin{n}{n}$,
and the Grassmannian $\Gr(n,2n)$ is the
subvariety of $\P(\bigwedge^n\C^{2n})$ defined
by the Pl\"ucker relations.  

\begin{prop}\cite{Ful97,Ho43}
For $\alpha,\beta\in\angbin{n}{n}$ there is a Pl\"ucker
relation
\begin{eqnarray*}
p_{\alpha}p_{\beta}-p_{\alpha\wedge\beta}p_{\alpha\vee\beta}
& + &
\sum_{\gamma\leq\alpha\wedge\beta<\alpha\vee\beta\leq\delta}
 c_{\alpha,\beta}^{\gamma,\delta}p_{\gamma}p_{\delta}=0\,.
\end{eqnarray*}
The defining ideal of $\Gr(n,2n)\subseteq\P(\bigwedge^n\C^{2n})$
is generated by the Pl\"ucker relations.
\end{prop}

Fix an ordered basis
$\{e_{\bar{n}},\dotsc,e_{\bar1},e_1,\dotsc,e_n\}$
of the vector space $\C^{2n}$, and let
$\Omega:=\sum_{i=1}^ne_{\bari}\wedge e_i$
be a non-degenerate alternating bilinear form.
The {\it Lagrangian Grassmannian} $\LG(n)$ is the set of
maximal isotropic subspaces of $\C^{2n}$ (relative to $\Omega$).

Let $\{h_i:=E_{ii}-E_{\bari\bari}\mid i\in [n]\}$
be the usual basis for the Lie algebra
$\mathfrak{t}$ of $T$~\cite{FH}, and let
$\{h^*_i\mid i\in [n]\}\subseteq\mathfrak{t}^*$ the dual basis.
Observe that $h_{\bari}^*=-h_i^*$.
The weights of any representation of
$\Sp_{2n}(\C)$ are $\Z$-linear combinations
of the fundamental weights
$\omega_i=h_{n{-}i{+}1}^*+\dotsb +h_n^*$. 

The weights of both representations
${\rm L}(\omega_n)\subseteq\bigwedge^n\C^{2n}$
and $\bigwedge^n\C^{2n}$ are of the
form $\omega=\sum_{i=1}^n h_{\alpha_i}^*$ for some
$\alpha\in\tangbin{n}{n}$.  If $\alpha_j=\bar{\alpha}_{j'}$
for some $j,j'\in[n]$, then
$h_{\alpha_j}^*=-h_{\alpha_{j'}}^*$, and thus
the support of $\omega$ does not contain $h_{\alpha_j}^*$.  Hence
the set of all such weights $\omega$ are
indexed by elements $\alpha\in\tangbin{n}{k}$
($k=1,\dotsc,n$) which do not involve both
$i$ and $\bari$ for any $i=1,\dotsc,n$.

Let $V$ be a vector space. For simple alternating
tensors
$v:=v_1\wedge\dotsb\wedge v_l\in\bigwedge^lV$
and
$\varphi:=\varphi_1\wedge\dotsb\wedge\varphi_k\in\bigwedge^kV^*$,
there is a {\it contraction} defined by setting
\[
\varphi\contr v
:=
\left\{
 \begin{array}{cc}
  \sum_{I\in\binom{[l]}{k}}\pm
  v_1\wedge\dotsb\wedge\varphi_1(v_{i_1})\wedge\dotsb\wedge
  \varphi_k(v_{i_k})\wedge\dotsb\wedge v_l, & k\leq l\\
  0, & k>l
 \end{array}
\right.
\]
and extending bilinearly to a map
$\bigwedge^kV^*\otimes\bigwedge^lV
 \rightarrow\bigwedge^{l{-}k}V$.
In particular, for a fixed element
$\Phi\in\bigwedge^kV^*$, we obtain
a linear map
$\Phi\contr\bullet:
 \bigwedge^lV\rightarrow\bigwedge^{l{-}k}V$.

The Lagrangian Grassmannian embeds in $\P{{\rm L}(\omega_n)}$, where
${\rm L}(\omega_n)$ is the irreducible $\Sp_{2n}(\C)$-representation
of highest weight $\omega_n=h_1^*+\dotsb+h_n^*$.
By Proposition~\ref{ctxn.prop}, this representation is
isomorphic to the kernel of the contraction
$\Omega\contr\bullet:
 \bigwedge^{n}\C^{2n}\rightarrow\bigwedge^{n{-}2}\C^{2n}$.
We thus have a commutative diagram of injective maps:
\[
\begin{CD}
  \LG(n)     @>>>   \Gr(n,2n)\\
   @VVV                @VVV  \\
\P {\rm L}(\omega_n) @>>>\P(\bigwedge^n\C^{2n}).\\
\end{CD}\vspace{5pt}
\]
The next proposition implies that
$\LG(n)=\Gr(n,2n)\cap \P {\rm L}(\omega_n)$.

\begin{prop}\label{ctxn.prop}
The dual of the contraction map
\[
\Omega\contr\bullet:
 \bigwedge^n\C^{2n}\rightarrow\bigwedge^{n{-}2}\C^{2n}
\]
is the multiplication map
\[
\Omega\wedge\bullet:
 \bigwedge^{n{-}2}{\C^{2n}}^*\rightarrow\bigwedge^{n}{\C^{2n}}^*\,.
\]
Furthermore, the irreducible representation
${\rm L}(\omega_n)$ is defined
by the ideal generated by the linear forms
\[
L_n:={\rm span}\{\Omega\wedge
 e_{\alpha_1}^*\wedge\dotsb\wedge e_{\alpha_{n{-}2}}^*
 \mid\alpha\in\tangbin{n}{n{-}2}\}.
\]
These linear forms cut out $\LG(n)$
scheme-theoretically in $\Gr(n,2n)$.
Dually, ${\rm L}(\omega_n)=\ker(\Omega\contr\bullet)$.
\end{prop}

\begin{proof}
The proof of first statement is straightforward, and the second
can be found in~\cite[Ch. 3, Exercise 1; Ch. 6, Exercise 24]{Wey}.
\end{proof}

Since the linear forms spanning $L_n$
are supported on variables indexed
by $\alpha\in\angbin{n}{n}$ such that
$\{\bari,i\}\in\alpha$ for some $i\in[n]$,
the set of complementary
variables is linearly independent.
These are indexed by the set ${\mathcal P}_n$
of {\it admissible}
elements of $\tangbin{n}{n}$:
\begin{eqnarray*}
{\mathcal P}_n &:=&
 \bigl\{\alpha\in\tangbin{n}{n}\mid
 i\in\alpha\Leftrightarrow\bari\not\in\alpha\bigr\},
\end{eqnarray*}
and have a simple description in terms of
partitions (Proposition~\ref{sym.prop}).

Consider the lattice $\Z^2$ with coordinates $(a,b)$
corresponding to the point $a$ units to the right of
the origin and $b$ units below the
origin.  Given an increasing sequence
$\alpha\in\angbin{n}{n}$, let $[\alpha]$ be the
lattice path beginning at $(0,n)$, ending at $(n,0)$, and
whose $i^{th}$ step is vertical if $i\in\alpha$ and
horizontal if $i\not\in\alpha$.  We can associate a partition
to $\alpha$ by taking the boxes lying in the
region bounded by the coordinate axes and $[\alpha]$.
For instance, the sequence $\alpha=\bar4\bar223\in\angbin{4}{4}$
is associated to the partition shown in Figure~\ref{4223.ptn}.

\begin{figure}[htb]
\[
\begin{picture}(125,135)
\put(0,0){\includegraphics[scale=0.65]{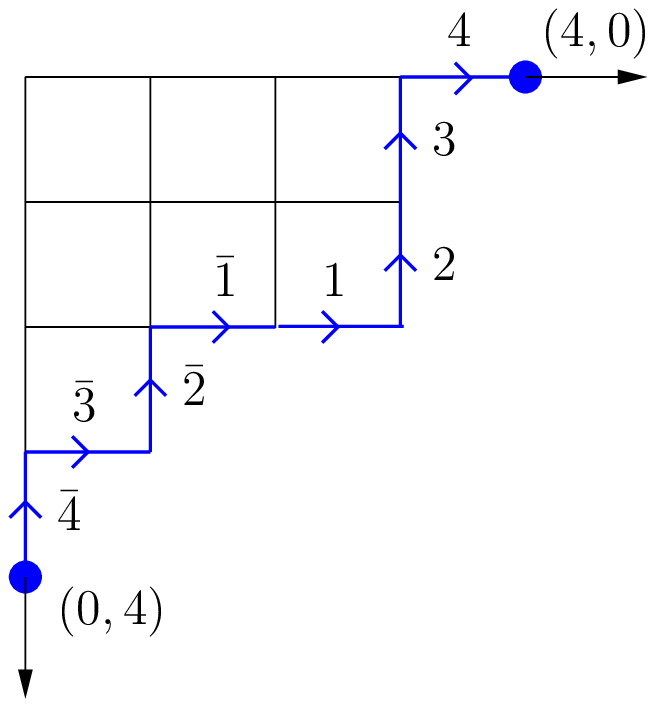}}
\end{picture}
\]
\caption{The partition $(3,3,1)$
associated to $\bar{4}\bar{2}23$.
\label{4223.ptn}}
\end{figure}

\begin{prop}\label{sym.prop}
The bijection between increasing sequences and partitions
induces a bijection between sequences $\alpha$ which
do not contain both $i$ and $\bari$ for any $i\in[n]$, and
partitions which lie inside the $n\times n$
square $(n^n)$ and are symmetric with respect to reflection
about the diagonal $\{(a,a)\mid a\in\Z\}\subseteq\Z^2$.
\end{prop}

\begin{proof}
The poset ${\mathcal P}_n$ consists
of those $\alpha\in\tangbin{n}{n}$
which are fixed upon negating
each element of $\alpha$ and
taking the complement in $\angles{n}$.
On the other hand, the
composition of these two
operations (in either order)
corresponds to reflecting the
associated diagram about the diagonal.
\end{proof}

\begin{rem}\label{negpos.rem}
We will use an element of
$\tangbin{n}{n}$ and its associated partition
interchangeably.  We denote by $\alpha^t$ the
{\it transpose} partition
obtained by reflecting $\alpha$
about the diagonal in $\Z^2$.
As a sequence, $\alpha^t$ is the complement
of $\{\bar\alpha_1,\dotsc,\bar\alpha_n\}\subseteq\angles{n}$.
We denote by $\alpha_+$ (respectively, $\alpha_-$)
the subsequence of positive (negative) elements
of $\alpha$.
\end{rem}
\begin{defn}\label{invol.adm.def}
The {\it Lagrangian involution} is the map
$\tau:p_{\alpha}\mapsto \sigma_{\alpha}p_{\alpha^t}$,
where $\sigma_{\alpha}:=
 {\rm sgn}(\alpha_+^c,\alpha_+)
  \cdot{\rm sgn}(\alpha_-,\alpha_-^c)=\pm 1$,
and ${\rm sgn}(a_1,\dotsc,a_s)$ denotes the sign of the permutation
sorting the sequence $(a_1,\dotsc,a_s)$.
\end{defn}
For example, if $\alpha=\bar4\bar123$,
then $\alpha_+=23$, $\alpha_-=\bar4\bar1$,
and $\sigma_{\alpha}=1$.

The Grassmannian $\Gr(n,2n)$ has a natural geometric
involution $\bullet^{\perp}:\Gr(n,2n)\rightarrow\Gr(n,2n)$
sending an $n$-plane
$U$ to its orthogonal complement
$U^{\perp}:=\{u\in \C^{2n}\mid
 \Omega(u,u')=0\mbox{{\rm, for all }}u'\in U\}$
with respect to $\Omega$.  The next proposition
relates $\bullet^{\perp}$ to the Lagrangian involution.

\begin{prop}\label{invol.prop}
The map $\bullet^{\perp}:\Gr(n,2n)\rightarrow\Gr(n,2n)$
expressed in Pl\"{u}cker coordinates
coincides with the Lagrangian involution:
\begin{eqnarray}
\Bigl[ p_{\alpha}~\Bigl\lvert~\alpha\in\tangbin{n}{n}\Bigr]
 & \mapsto &
 \Bigl[\sigma_{\alpha}p_{\alpha^t}
  ~\Bigr\rvert~\alpha\in\tangbin{n}{n}\Bigr]\,.
\end{eqnarray}
In particular, the relation
$p_{\alpha}-\sigma_{\alpha}p_{\alpha^t}=0$
holds on $\LG(n)$.
\end{prop}

\begin{proof}
The set of $n$-planes in $\C^{2n}$ which do not meet the
span of the first $n$ standard basis vectors is open and
dense in $\Gr(n,2n)$.  Any such $n$-plane is the row
space of an $n\times 2n$ matrix
\begin{eqnarray*}
Y &:=& \left(I\mid X\right)
\end{eqnarray*}
where $I$ is the $n\times n$ identity matrix
and $X$ is a generic $n\times n$ matrix.
We work in the affine coordinates
given by the entries in $X$.
For $\alpha\in\angbin{n}{n}$, denote the
$\alpha^{th}$ minor of $Y$ by
$p_{\alpha}(Y)$.
For a set of indices
$\alpha=\{\alpha_1,\dotsc,\alpha_k\}\subseteq[n]$,
let $\alpha^c:=[n]\setminus\alpha$
be the complement,
$\alpha':=\{n{-}\alpha_k{+}1,\dotsc,n{-}\alpha_1{+}1\}$,
and $\bar\alpha:=\{\bar{\alpha}_1,\dotsc,\bar{\alpha}_k\}$.
Via the correspondence between partitions and
sequences (Proposition~\ref{sym.prop}),
$\alpha^t=\bar{\alpha}^c$.

We claim that $\bullet^{\perp}$ reflects $X$ along
the antidiagonal.  To see this, we simply
observe how the rows of $Y$ pair under
$\Omega$.  For vectors $u,v\in\C^n$, let
$(u,v)\in\C^{2n}$ be the concatenation.  Let
$r_i:=(e_{i},v_{i})\in\C^{2n}$
be the $i^{th}$ row of $Y$.  
For $k\in\angles{n}$, we let
$r_{ik}\in\C$ be the $k^{th}$ entry of $r_i$.
Then, for $i,j\in\angles{n}$,
\begin{eqnarray*}
\Omega(r_i,r_j) &=&
(e_i,v_i)\cdot(-v_j,e_j)^t\\
&=&
 r_{i,n{-}j{+}1}-r_{j,n{-}i{+}1}\,.
\end{eqnarray*}
It follows that the effect
of $\bullet^{\perp}$ on the minor 
$X_{\rho,\gamma}$ of $X$
given by row indices $\rho$ and column indices $\gamma$ is
\begin{eqnarray*}
(X^{\perp})_{\rho,\gamma} &=&
  X_{\gamma',\rho'}\,.
\end{eqnarray*}
Let $\alpha=\bar\epsilon\cup\phi\in\tangbin{n}{n}$, where
$\epsilon$ and $\phi$ are subsets of $[n]$
whose cardinalities sum to $n$.
Combining the above description of $\bullet^{\perp}$
with the identity
\[
 p_{\alpha}(Y)=
  {\rm sgn}(\epsilon^c,\epsilon)
  X_{(\epsilon^c)',\phi}\,,
\]
we have
\begin{eqnarray*}
p_{\alpha}(Y^{\perp})&=& {\rm sgn}(\epsilon^c,\epsilon)
         (X^{\perp})_{(\epsilon^c)',\phi}\\
 &=& {\rm sgn}(\epsilon^c,\epsilon)
         X_{\phi',\epsilon^c}\\
 &=& {\rm sgn}(\epsilon^c,\epsilon){\rm sgn}(\phi,\phi^c)
                p_{(\bar{\phi}^c,\epsilon^c)}(Y)\\
 &=& \sigma_{\alpha}p_{\alpha^t}(Y)\,.
\end{eqnarray*}
It follows that the relation
$p_{\alpha}-\sigma_{\alpha}p_{\alpha^t}=0$
holds on a dense Zariski-open subset
and hence identically on all of $\LG(n)$.
\end{proof}

By Proposition~\ref{invol.prop}, the system
of linear forms
\begin{eqnarray}\label{binom.eqn}
L_n'&:=&{\rm span}\{ p_{\alpha}-\sigma_{\alpha}p_{\alpha^t}
 \mid\alpha\in\tangbin{n}{n}\}
\end{eqnarray}
defines $\LG(n)\subseteq\Gr(n,2n)$ {\it set}-theoretically.
Since $\LG(n)$ lies in no hyperplane of $\P {\rm L}(\omega_n)$,
$L_n'$ is a linear subspace of the span $L_n$ of the
defining equations of ${\rm L}(\omega_n)\subseteq\bigwedge^n\C^{2n}$.
The generators of $L'_n$ given in~(\ref{binom.eqn})
suggest that homogeneous coordinates
for the Lagrangian Grassmannian should be
indexed by some sort of quotient (which we will call
${\mathcal D}_n$) of the poset ${\mathcal P}_n$.  The correct
notion is that of a doset
(Definition~\ref{doset.def}).
An important set of representatives for ${\mathcal D}_n$
in ${\mathcal P}_n$ is the set of {\it Northeast}
partitions (Proposition~\ref{fiber.prop}).

\begin{rem}\label{strict.rem}
 The set of {\it strict} partitions with at most $n$
 rows and columns is commonly used to index
 Pl\"ucker coordinates for the Lagrangian Grassmannian.
 Given a symmetric partition $\alpha\in{\mathcal P}_n$,
 we can obtain a strict partition by first removing the
 boxes of $\alpha$ which lie below the diagonal, and
 then left-justifying the remaining boxes.  This gives
 a bijection between the two sets of partitions.
\end{rem}

By Proposition~\ref{sym.prop}, we may identify elements of
$\tangbin{n}{n}$ with partitions lying in the $n\times n$
square $(n^n)$, and ${\mathcal P}_n$ with the set of
symmetric partitions.
Define a map
$\pi_n:\tangbin{n}{n}
 \rightarrow{\mathcal P}_n\times{\mathcal P}_n$
by
$\pi_n(\alpha):=(\alpha\wedge\alpha^t,\alpha\vee\alpha^t)$.
Let ${\mathcal D}_n$ be the image of $\pi_n$.
It is called the set of {\it admissible pairs},
and is a subset of
${\mathcal O}_{{\mathcal P}_n}:=
 \{(\alpha,\beta)\in{\mathcal P}_n\times{\mathcal P}_n
 \mid\alpha\leq\beta\}$.
The image of ${\mathcal P}_n\subseteq\tangbin{n}{n}$
under $\pi_n$ is the diagonal
$\Delta_{{\mathcal P}_n}\subseteq{\mathcal P}_n\times{\mathcal P}_n$.

To show that ${\mathcal D}_n$ indexes coordinates on $\LG(n)$,
we will work with a convenient set of
representatives of the fibers of $\pi_n$.
The fiber over $(\alpha,\beta)\in{\mathcal D}_n$ can be
described as follows.  The lattice paths $[\alpha]$ and
$[\beta]$ must meet at the diagonal.  Since $\alpha$
and $\beta$ are symmetric, they are determined by the segments
of their associated paths to the right and above the diagonal.
Let $\Pi(\alpha,\beta)$ be the set
of boxes bounded by these segments.
Taking $n=4$ for example, $\Pi(\bar4\bar213,\bar3\bar124)$
consists of the two shaded boxes above the diagonal in
Figure~\ref{4321.321.fig}.  The lattice path
$[\bar4\bar213]$ is above and to the left
of the path $[\bar3\bar124]$.
\begin{figure}[htb]
\[
\begin{picture}(100,80)(0,0)
\put(32,0){\includegraphics[scale=0.2]{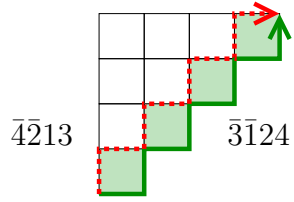}}
\put(0,20){$\bar4\bar213$}
\put(82,20){$\bar3\bar124$}
\end{picture}
\]
\caption{The paths associated to $\bar4\bar213$ and
         $\bar3\bar124$ in ${\mathcal P}_4$.}
\label{4321.321.fig}
\end{figure}

For any partition $\alpha\subseteq(n^n)$,
the set $\alpha_{+}\subseteq\alpha$
consists of the boxes of $\alpha$ on or
above the main diagonal, and
$\alpha_{-}\subseteq\alpha$ consists
of the boxes of $\alpha$ on or below
the main diagonal ({\it cf.} Remark~\ref{negpos.rem}).
Similarly, let
$\Pi_+(\alpha,\beta)\subseteq\Pi(\alpha,\beta)$
be the set of boxes above the diagonal and let
$\Pi_-(\alpha,\beta)\subseteq\Pi(\alpha,\beta)$
be the set of boxes below the diagonal.

A subset
$S\subseteq(n^n)$ of boxes is {\it disconnected}
if $S=S'\sqcup S''$ and no box of $S'$ shares
an edge with a box of $S''$.  A subset
$S$ is connected if it is not disconnected.  Let
$\bigsqcup_{i=1}^kS_i$ be the decomposition of
$\Pi_+(\alpha,\beta)$ into its connected components
(so that $\bigsqcup_{i=1}^kS_i^t$ is the decomposition
of $\Pi_-(\alpha,\beta)$).
\begin{figure}[htb]
\[
\begin{picture}(200,100)(0,0)
\put(0,20){\includegraphics[scale=0.2]{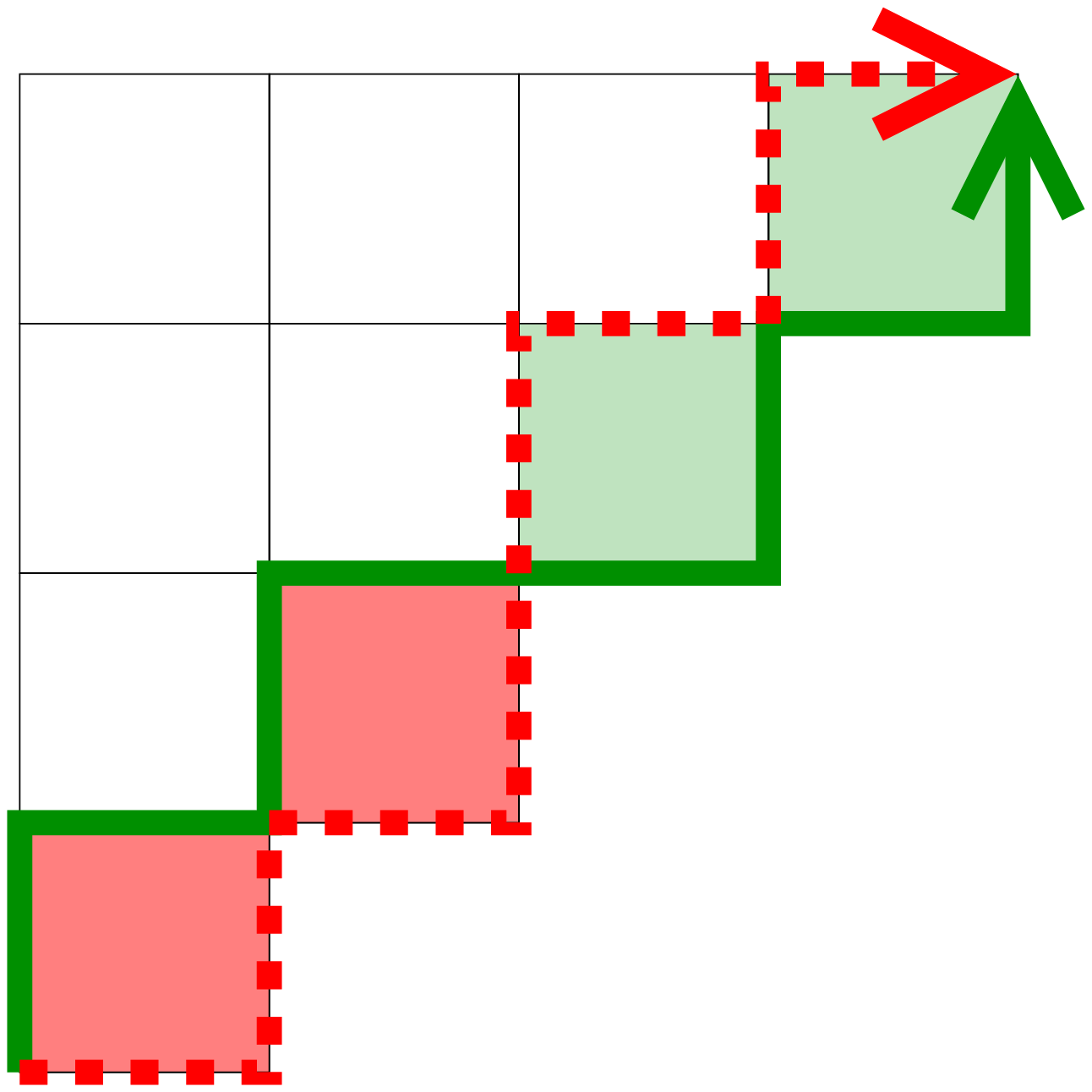}}
\put(0,0){$\bar4\bar224$ and $\bar3\bar113$}
\put(140,20)
 {\includegraphics[scale=0.2]{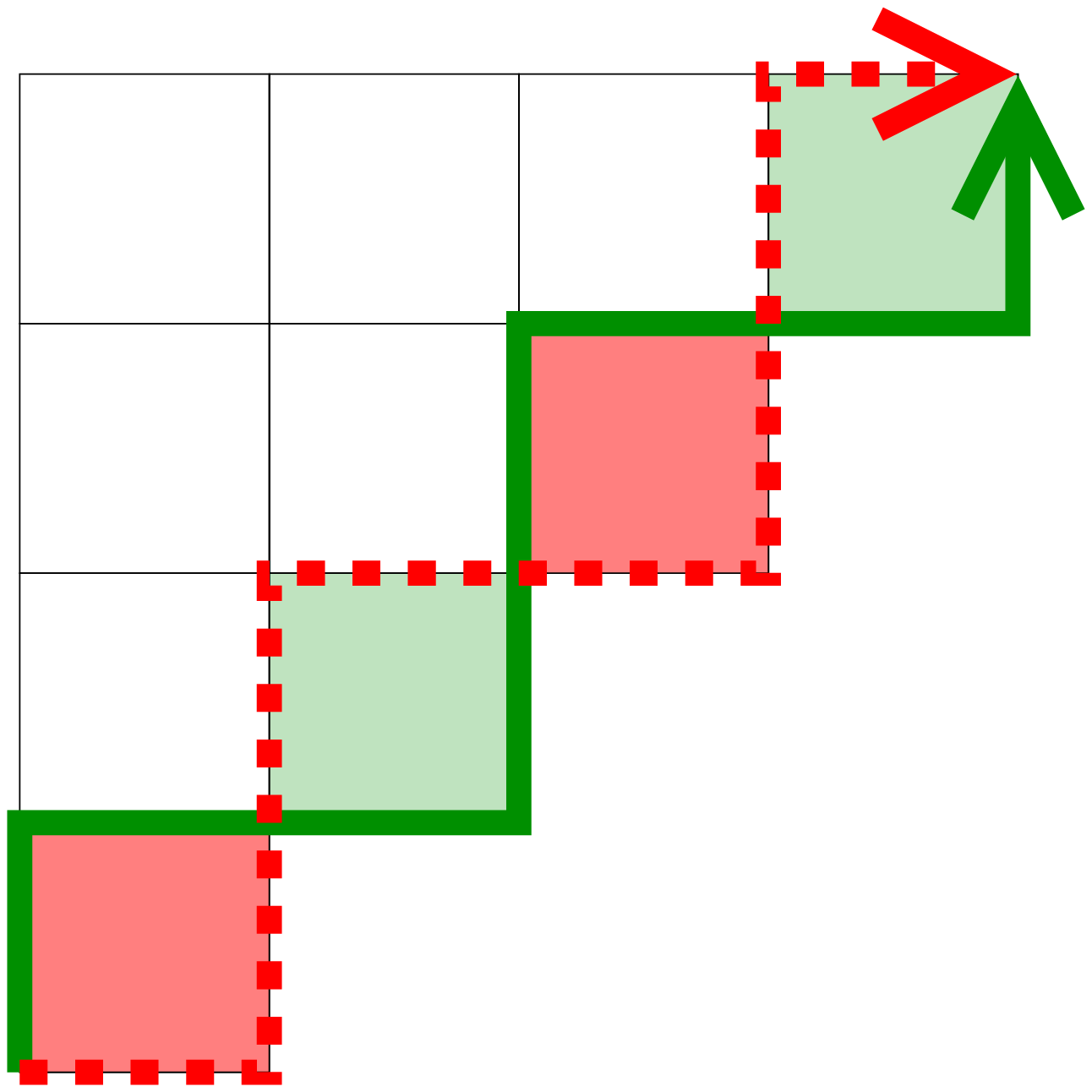}}
\put(140,0){$\bar4\bar114$ and $\bar3\bar223$}
\end{picture}
\]
\caption{Elements of $\pi_n^{-1}(\bar4\bar213,\bar3\bar124)$.
\label{4321.321.fiber.fig}}
\end{figure}

Any element $\gamma$ of the fiber $\pi_n^{-1}(\alpha,\beta)$
is obtained by choosing a subset $I\subseteq[k]$ and setting
\begin{eqnarray*}
\gamma &=& \alpha\cup\Bigl(\bigcup_{i\not\in I}S_i^t\Bigr)\cup
 \Bigl(\bigcup_{i\in I}S_i\Bigr).
\end{eqnarray*}
The elements of 
$\pi_n^{-1}(\bar4\bar213,\bar3\bar124)$ are shown in
Figure~\ref{4321.321.fiber.fig}.

\begin{defn}\label{ne.def}
A partition $\alpha$ {\it Northeast} if
$\alpha_-^t\subseteq\alpha_+$ and
{\it Southwest} if its transpose is Northeast.
\end{defn}

For example, $\bar4\bar224\in{\mathcal P}_4$
is Northeast while
$\bar4\bar114\in{\mathcal P}_4$
is neither Northeast nor Southwest.  
We summarize these ideas in the following proposition.

\begin{prop}\label{fiber.prop}
Let $(\alpha,\beta)\in{\mathcal D}_n$ be an element of the
image of $\pi_n$.  Then $\pi_n^{-1}(\alpha,\beta)$ is in
bijection with the set of subsets of
connected components of $\Pi_+(\alpha,\beta)$.
There exists a unique Northeast element
of $\pi_n^{-1}(\alpha,\beta)$;
namely, the element corresponding to all the connected
components.  Similarly, there is a unique Southwest
element corresponding to the empty set of components.
\end{prop}

\begin{ex}\label{lg4}
The Lagrangian Grassmannian $\LG(4)\subseteq\Gr(4,8)$
is defined by the ideal $L_4$.  From the explicit
linear generators given in Proposition~\ref{ctxn.prop},
it is evident that
$L_4\subseteq\bigwedge^4{\C^8}^*$ is spanned by weight
vectors.  For example, the generators of
$L_4\cap(\bigwedge^4{\C^8}^*)_0$ are
the vectors of weight zero:
\begin{equation}\label{lg4lin}
\begin{array}{ccccccc}
\Omega\wedge{p_{\bar11}} & = &
 p_{\bar4\bar114}&+&p_{\bar3\bar113}&+&p_{\bar2\bar112}\\
\Omega\wedge{p_{\bar22}} & = &
 p_{\bar4\bar224}&+&p_{\bar3\bar223}&+&p_{\bar2\bar112}\\
\Omega\wedge{p_{\bar33}} & = &
 p_{\bar4\bar334}&+&p_{\bar3\bar223}&+&p_{\bar3\bar113}\\
\Omega\wedge{p_{\bar44}} & = &
 p_{\bar4\bar334}&+&p_{\bar4\bar224}&+&p_{\bar4\bar114}\\
\end{array}
\end{equation}
The following linear forms lie in the span of the right-hand
side of~(\ref{lg4lin}):
\begin{equation}\label{lg4bin}
\begin{array}{ccc}
p_{\bar2\bar112}  +  p_{\bar4\bar334} & 
p_{\bar3\bar113}  +  p_{\bar4\bar224} & 
p_{\bar3\bar223} + p_{\bar4\bar114}\\[5pt]
&
p_{\bar4\bar114} + p_{\bar4\bar224} + p_{\bar4\bar334}
&
\end{array}
\end{equation}
Three of the linear forms in~(\ref{lg4bin})
are supported on a pair $\{p_{\alpha},p_{\alpha^t}\}$,
and the remaining linear form expresses
the Pl\"ucker coordinate $p_{\bar4\bar114}$
as a linear combination of
coordinates indexed by Northeast partitions
(in general, this follows from
Lemma~\ref{basis.lemma}).

Since each pair 
$\{p_{\alpha},p_{\alpha^t}\}$ is
incomparable, there is a Pl\"ucker
relation which, after reduction
by the linear forms~(\ref{lg4lin}),
takes the form
\[
 \pm p_{\alpha}^2-p_{\beta}p_{\gamma}+\text{lower order terms}
\]
where $\beta:=\alpha\wedge\alpha^t$
and $\gamma:=\alpha\vee\alpha^t$ are respectively the
meet and join of $\alpha$ and $\alpha^t$. Defining
$p_{(\beta,\gamma)}:=p_{\alpha}=\sigma_{\alpha} p_{\alpha^t}$
we can regard such an equation as giving a rule
for rewriting $p_{(\beta,\gamma)}^2$ as a linear combination
of monomials supported on a chain.  These general case
is treated in Section~\ref{str8law.sec}.
\end{ex}
%
\section{Algebras with straightening law}\label{asl.sec}

\subsection{Generalities}
The following definitions are fundamental.

Let ${\mathcal P}$ be a poset,
$\Delta_{\mathcal P}$
the diagonal in ${\mathcal P}\times{\mathcal P}$,
and ${\mathcal O}_{\mathcal P}:=
\{(\alpha,\beta)\in{\mathcal P}\times{\mathcal P}
\mid\alpha\leq\beta\}$ the subset of
${\mathcal P}\times{\mathcal P}$
defining the order relation on ${\mathcal P}$.

\begin{defn}\label{doset.def}
A {\it doset} on ${\mathcal P}$
is a set ${\mathcal D}$ such that $\Delta_{\mathcal P}\subseteq
{\mathcal D}\subseteq{\mathcal O}_{\mathcal P}$, and if
$\alpha\leq\beta\leq\gamma$, then $(\alpha,\gamma)\in{\mathcal D}$
if and only if $(\alpha,\beta)\in{\mathcal D}$ and
$(\beta,\gamma)\in{\mathcal D}$.  The ordering on ${\mathcal D}$
is given by $(\alpha,\beta)\leq(\gamma,\delta)$
if and only if $\beta\leq\gamma$ in
${\mathcal P}$.  We call ${\mathcal P}$
the {\it underlying poset}.
\end{defn}
\begin{rem}
A doset is not usually a poset: the ordering is
non-reflexive except in the trivial case when
${\mathcal D}=\Delta_{{\mathcal P}}$.  In any
case, we may regard
${\mathcal P}\cong\Delta_{{\mathcal P}}$
as a subset of ${\mathcal D}$.
\end{rem}
The {\it Hasse diagram} of a doset
${\mathcal D}$ on ${\mathcal P}$
is obtained from the Hasse diagram
of ${\mathcal P}\subseteq{\mathcal D}$ by
drawing a double line for
each cover $\alpha\lessdot\beta$ such that
$(\alpha,\beta)$ is in ${\mathcal D}$.
The defining property of a doset implies that
we can recover all the information in the doset
from its Hasse diagram.  See Figure~\ref{D24.fig}
in Section~\ref{admpairs.ssec} for an example.

An algebra with straightening law is
an algebra generated by indeterminates
$p_{\alpha}$, $\alpha\in{\mathcal D}\}$
indexed by a (finite)
doset ${\mathcal D}$ with a basis consisting
of {\it standard} monomials
supported on a chain.
That is, a monomial
$p_{(\alpha_1,\beta_1)}\dotsb p_{(\alpha_k,\beta_k)}$
is standard if $\alpha_1\leq\beta_1\leq\alpha_2\leq\beta_2
\leq\dotsb\leq\alpha_k\leq\beta_k$.
Furthermore, monomials which are not
standard are subject to
certain {\it straightening relations},
as described in the following definition.

\begin{defn}\label{asl.def}
Let ${\mathcal D}$ be a doset.
A graded $\C$-algebra $A=\bigoplus_{q\geq 0}A_q$ is an
{\it algebra with straightening law} on ${\mathcal D}$
if there is an injection
${\mathcal D}\owns(\alpha,\beta)
 \mapsto p_{(\alpha,\beta)}\in A_1$
such that:

\begin{enumerate}
\item 
      The set $\{p_{(\alpha,\beta)}\mid
      (\alpha,\beta)\in{\mathcal D}\}$
      generates $A$.
\item 
      The set of standard monomials are a $\C$-basis of $A$.
\item
      For any monomial
      $m=p_{(\alpha_1,\beta_1)}\dotsb p_{(\alpha_k,\beta_k)}$,
      $(\alpha_i,\beta_i)\in{\mathcal D}$ and $i=1,\dotsc,k$, if
\[
m=\sum_{j=1}^N
 c_jp_{(\alpha_{j1},\beta_{j1})}\dotsb p_{(\alpha_{jk},\beta_{jk})},
\]
     is the unique expression of $m$
     as a linear combination of
     distinct standard monomials,
     then the sequence
$(\alpha_{j1}\leq\beta_{j1}\leq\dotsb\leq\alpha_{jk}\leq\beta_{jk})$
     is lexicographically smaller than
$(\alpha_1\leq\beta_1\leq\dotsb\leq\alpha_k\leq\beta_k)$.
     That is, if $\ell\in[2k]$ is minimal such
     that $\alpha_{j\ell}\neq\alpha_{\ell}$,
     then $\alpha_{j\ell}<\alpha_{\ell}$.
\item
      If $\alpha_1\leq\alpha_2\leq\alpha_3\leq\alpha_4$
      are such that for some permutation
      $\sigma\in S_4$ we have
      $(\alpha_{\sigma(1)},\alpha_{\sigma(2)})\in{\mathcal D}$
      and
      $(\alpha_{\sigma(3)},\alpha_{\sigma(4)})\in{\mathcal D}$,
      then
\begin{eqnarray}\label{perm.str.rel}
 p_{(\alpha_{\sigma(1)},\alpha_{\sigma(2)})}
  p_{(\alpha_{\sigma(3)},\alpha_{\sigma(4)})}
 =\pm p_{(\alpha_1,\alpha_2)}p_{(\alpha_3,\alpha_4)}
  +\sum_{i=1}^Nr_im_i
\end{eqnarray}
     where the $m_i$ are quadratic
     standard monomials distinct from
     $p_{(\alpha_1,\alpha_2)}p_{(\alpha_3,\alpha_4)}$.
\end{enumerate}
\end{defn}

The ideal of straightening relations is generated
by homogeneous quadratic forms
in the $p_{\alpha}$ ($\alpha\in{\mathcal D}$),
so we may consider the projective variety
$X:={\rm Proj}~A$ they define.
For each $\alpha\in{\mathcal P}$,
we have the {\it Schubert variety}
\[
X_{\alpha}:=\{x\in X\mid
 p_{(\beta,\gamma)}(x)=0 \text{ for } \gamma\not\leq\alpha\}
\]
and the {\it dual Schubert variety}
\[
X^{\alpha}:=\{x\in X\mid
 p_{(\beta,\gamma)}(x)=0 \text{ for } \beta\not\geq\alpha\}.
\]

\begin{rem}
If $\alpha\lessdot\beta$ and $(\alpha,\beta)\in{\mathcal D}$,
then the multiplicity of $X_{\alpha}$ in $X_{\beta}$ is $2$,
and likewise for the multiplicity of $X^{\alpha}$ in $X^{\beta}$.
This fact will arise in Section~\ref{enumgeom.sec} when we consider
enumerative questions.
\end{rem}

We recall the case when $X$ is the
Grassmannian of $k$-planes in
$\C^n$, whose coordinate ring is an algebra
with straightening law on the poset $\binom{[n]}{k}$.

For each $i\in[n]$, set $F_i:=\langle e_1,\dotsc,e_i\rangle$
and $F'_i:=\langle e_n,\dotsc,e_{n{-}i{+}1}\rangle$,
where $\langle\,\dotsb\,\rangle$ denotes linear span and
$\{e_1,\dotsc,e_n\}$ is the standard basis of $\C^n$.
We call $\Fdot:=\{F_1\subseteq\dotsb\subseteq F_n\}$ the
{\it standard coordinate} flag, and
$\Fdot':=\{F'_1\subseteq\dotsb\subseteq F'_n\}$ the
{\it opposite} flag.

For
$\alpha\in\binom{[n]}{k}$,
let
$X_{\alpha}:=
 \{E\in\Gr(k,n)\mid\dim(E\cap F_{\alpha_i})\geq i,~{\text{for}}~i=1,\dotsc,k\}$
be the {\it Schubert variety} indexed by $\alpha$ and
$X^{\alpha}:=
 \{E\in\Gr(k,n)\mid
  \dim(E\cap F'_{n{-}\alpha_i{+}1})\geq k{-}i{+}1,~{\text{for}}~i=1,\dotsc,k\}$
the {\it dual Schubert variety}.

We represent any $k$-plane $E\in\Gr(k,n)$
as the row space of a $k\times n$ matrix.
Furthermore, any such $k$-plane $E$ is
the row space of a unique reduced row
echelon matrix.  The Schubert variety
$X_{\alpha}$ consists of
precisely the $k$-planes $E$
such that the pivot in row $i$ is weakly
to the left of column $\alpha_i$.  Since the Pl\"ucker
coordinate $p_{\beta}(E)$ is just
the $\beta^{th}$ maximal minor of
this matrix, we see that
$E\in X_{\alpha}$ if and only if
$p_{\beta}(E)=0$ for all $\beta\not\leq\alpha$; hence
our definition of the Schubert variety
$X_{\alpha}$ (and by a similar argument, the
dual Schubert variety $X^{\alpha}$)
for the Grassmannian agrees with
the general definition above.

For a fixed projective variety $X\subseteq\P^n$,
there are many homogeneous ideals
which cut out $X$ set-theoretically.
However, there exists a unique
such ideal which is saturated and radical.
Under mild hypotheses, any ideal
generated by straightening relations
on a doset is saturated and radical.
The proofs of Theorems~\ref{saturated.thm}
and~\ref{reduced.thm} illustrate
the usefulness of Schubert varieties in the study of
an algebra with straightening law.

\begin{defn}
An ideal $I\subseteq\C[x_0,\dotsc,x_n]$
is {\it saturated} if, given a polynomial
$f\in\C[x_0,\dotsc,x_n]$ and an integer $N\in\N$,
$x_i^Nf=0\mod~I$ for all
$i=0,\dotsc,n$ implies that
$f=0\mod~I$.
\end{defn}

\begin{defn}
An ideal $I\subseteq\C[x_0,\dotsc,x_n]$
is {\it radical} if, given a plynomial
$f\in\C[x_0,\dotsc,x_n]$ and an integer $N\in\N$,
$f^N=0\mod~I$ implies that $f=0\mod~I$.
\end{defn}

The next two theorems concern an algebra with
straightening law on a doset ${\mathcal D}$
with underlying poset ${\mathcal P}$.  We 
write $A=\C[{\mathcal D}]/J$ for this algebra,
where $J$ is the ideal generated by the
straightening relations.

\begin{thm}\label{saturated.thm}
Let ${\mathcal D}$ be a doset whose underlying poset
has a unique minimal element $\alpha_0$.  Then any
ideal $J$ of straightening relations on ${\mathcal D}$
is saturated.
\end{thm}

\begin{proof}
Let $f\not\in J$.  Modulo $J$, we may write
$f=\sum_{i=1}^ka_im_i$ where the $m_i$ are
(distinct) standard monomials, and $a_i\in\C$.
For each $N\in\N$,
$p_{\alpha_0}^Nf=\sum_{i=1}^ka_ip_{\alpha_0}^Nm_i$ is
a linear combination of standard monomials, since
${\rm supp}~m_i\cup\{\alpha_0\}$ is a chain
for each $i\in[k]$.  It is non-trivial since
$p_{\alpha_0}^Nm_i=p_{\alpha_0}^Nm_j$
implies $i=j$.  Thus $p_{\alpha_0}^Nf\not\in J$ for any
$N\in\N$.
\end{proof}

\begin{thm}\label{reduced.thm}
An algebra with straightening law on a doset is reduced.
\end{thm}

\begin{proof}
Let $A$ be an algebra with straightening law.
For $f\in A$ and $\alpha\in{\mathcal P}$, denote by
$f_{\alpha}$ the restriction of $f$ to the
dual Schubert variety $X^{\alpha}$.

We will show by induction on the poset ${\mathcal P}$
that $f_{\alpha}^n=0$ implies $f_{\alpha}=0$.  Note that
by induction on $n$ it suffices to do this for $n=2$.
Indeed, assume that we have shown that $f^2=0$
implies $f=0$ for any $f$ in some ring $A$, and
suppose $f^n=0$.  Then $(f^{\lceil\frac{n}{2}\rceil})^2=0$,
so that $f^{\lceil\frac{n}{2}\rceil}=0$ by our assumption,
and thus $f=0$ by induction.

Let $f\in A$ be such that $f_{\alpha}^2=0$.
In particular, $f_{\beta}^2=0$ for all $\beta\geq\alpha$
(since $X^{\beta}\subseteq X^{\alpha}$), so that
$f_{\beta}=0$ by induction.  It follows that $f_\alpha$ is
supported on monomials on $X^\alpha$ which vanish on
$X^\beta$ for all $\beta\geq\alpha$.  That is,
\begin{eqnarray}\label{sum.std.mons}
f_{\alpha} &=& \sum_{i=1}^mc_i p_{\alpha}^{e_i}
 p_{(\alpha,\beta_{1,i})}\dotsb p_{(\alpha,\beta_{\ell_i,i})}\,.
\end{eqnarray}
For the right hand side
of~(\ref{sum.std.mons})
to be standard, we must
have $\ell_i=1$ for all
$i=1,\dotsc m$.  Also, homogeneity implies that
$e:=e_1=\dotsb=e_m$ for $i=1,\dotsc m$.
Thus, if we set $\beta_i:=\beta_{1,i}$,
then $f_{\alpha}$ has the form
\begin{eqnarray}\label{sum.std.mons2}
f_{\alpha} &=& p_{\alpha}^e\sum_{i=1}^mc_i
 p_{(\alpha,\beta_i)}\,.
\end{eqnarray}
Choose a linear extension of ${\mathcal D}$
as follows.  Begin with a linear extension
of ${\mathcal P}\subseteq{\mathcal D}$.
For incomparable
elements $(\alpha,\beta),(\gamma,\delta)$ of
${\mathcal D}$, set $(\alpha,\beta)\leq(\gamma,\delta)$
if $\beta<\delta$ or $\beta=\delta$ and
$\alpha\leq\gamma$.
With respect to the resulting
linear ordering of the variables,
take the lexicographic term order on
monomials in $A$.

For an element
$g\in A$, denote by ${\rm lt}(g)$ (respectively,
${\rm lm}(g)$) the lead term (respectively, lead monomial)
of $g$.  Reordering the terms in~(\ref{sum.std.mons2}) if
necessary, we may assume that
${\rm lt}(f_{\alpha})=c_1p_{\alpha}^ep_{(\alpha,\beta_1)}$.

Writing $f_{\alpha}^2$ as a linear combination of
standard monomials (by first expanding the square of the
right hand side of~(\ref{sum.std.mons2})
and then applying the
straightening relations), we see that
${\rm lt}(f_{\alpha}^2)=
 \pm c_1^2p_{\alpha}^{2e{+}1}p_{\beta_1}$.
This follows from our choice of term order and the fourth
condition in Definition~\ref{asl.def}.

We claim that
${\rm lt}(f_{\alpha}^2)$
cannot be cancelled in the expression for $f_{\alpha}^2$
as a sum of standard monomials.  Indeed, suppose there
are $i,j\in[m]$ such that
${\rm lm}((p_{\alpha}^ep_{(\alpha,\beta_i)})\cdot
 (p_{\alpha}^ep_{(\alpha,\beta_j)}))
 =p_{\alpha}^{2e{+}1}p_{\beta_1}$.
Then by the straightening relations,
$\beta_1\leq\beta_i,\beta_j$.
But $\beta_1\not<\beta_i$
since ${\rm lm}(f_{\alpha})=p_{\alpha}^ep_{(\alpha,\beta_1)}$.
For the same reasons, $\beta_1\not<\beta_j$.
Therefore $\beta_i=\beta_j=\beta_1$, so
$c_1p_{\alpha}^ep_{(\alpha,\beta_1)}$ is
the only term contributing to the monomial
$p_{\alpha}^{2e{+}1}p_{\beta_1}$ in $f_{\alpha}^2$.
\end{proof}
\begin{rem}
Theorem~\ref{reduced.thm} was first proved for
an algebra with straightening law on a {\it poset}
in~\cite{Eis79}, but the methods used (deformation
to the initial ideal) are not well-suited for a doset.
The proof given here is essentially an
extension of the proof of Bruns and
Vetter~\cite[Theorem 5.7]{BV}
to the doset case.
\end{rem}

\subsection{Hilbert series of an algebra with straightening law}
\label{hilb.ssec}

We compute the
Hilbert series of an algebra with straightening law
$A$ on a doset, and thus obtain
formulas for the dimension
and degree of ${\rm Proj}~A$.
Let ${\mathcal P}$ be a poset and
${\mathcal D}$ a doset on ${\mathcal P}$.
Assume that ${\mathcal P}$ and ${\mathcal D}$
are {\it ranked}; that is, any two maximal chains in
${\mathcal D}$ (respectively, ${\mathcal P}$)
have the same length.  Define ${\rm rank}~{\mathcal D}$
(respectively, ${\rm rank}~{\mathcal P}$) to be the
length of any maximal chain in ${\mathcal D}$
(respectively, ${\mathcal P}$).

First, we compute the Hilbert series of $A$ with
respect to a suitably chosen fine grading of $A$
by the elements of a semigroup, as follows.

Monomials in $\C[{\mathcal D}]$ are determined by their
exponent vectors.  We can therefore identify
the set of such monomials with the semigroup
$\N^{\mathcal D}$.  Define the {\it weight map} 
${\bf w}:\N^{\mathcal D}\rightarrow\QQ^{\mathcal P}$
by setting
${\bf w}(\alpha,\beta):=
 \frac{\epsilon_{\alpha}+\epsilon_{\beta}}{2}$,
where $\epsilon_{\alpha}\in\QQ^{\mathcal P}$
($\alpha\in{\mathcal P}$) is the vector
with $\alpha$-coordinate equal to $1$
and all other coordinates equal to $0$.
This gives a grading
of $A$ by the semigroup ${\rm im}({\bf w})$.
Let ${\rm Ch}({\mathcal D})$
be the set of all chains in ${\mathcal D}$.
Since the standard monomials
(those supported on a chain)
form a $\C$-basis for $A$,
the Hilbert series with respect
to this fine grading is
\[
H_A(r)=
\sum_{c\in{\rm Ch}({\mathcal D})}
\sum_{\substack{a\in{\rm im}({\bf w})\\{\rm supp}(a)=c}}r^a
\]
where
$r:=(r_{\alpha}\mid\alpha\in{\mathcal P})$,
$a=(a_{\alpha}\mid\alpha\in{\mathcal P})$, and
$r^a=\prod_{\alpha\in{\mathcal P}}r_{\alpha}^{a_{\alpha}}$.
Note that elements of
${\rm im}({\bf w})$ correspond to
certain monomials with {\it rational}
exponents (supported on ${\mathcal P}$).
For example, $(\alpha,\beta)\in{\mathcal D}$
corresponds to
$\sqrt{r_{\alpha}r_{\beta}}$. 
Setting all
$r_{\alpha}=r$,
we obtain the usual (coarse) Hilbert series,
defined with respect to the
usual $\Z$-grading on $A$ by degree.

\begin{ex}\label{barbell.ex}
Consider the doset ${\mathcal D}:=\{\alpha,(\alpha,\beta),\beta\}$
on the two element poset $\{\alpha<\beta\}$.  The elements of
${\rm Ch}({\mathcal D})$ are shown in Figure~\ref{chains.fig}.
\begin{eqnarray*}
{\rm Ch}({\mathcal D})&=&
\{\emptyset,\{\alpha\},\{\beta\},
\{(\alpha,\beta)\},\{\alpha,(\alpha,\beta)\},
\{(\alpha,\beta),\beta\},\{\alpha,\beta\},
\{\alpha,(\alpha,\beta),\beta\}\}\,.
\end{eqnarray*}
\begin{figure}[htb]
\[
\begin{picture}(200,115)
\put(0,93){\includegraphics[scale=0.3]{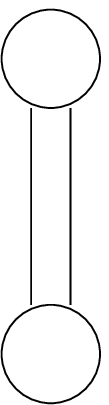}}
\put(60,93){\includegraphics[scale=0.3]{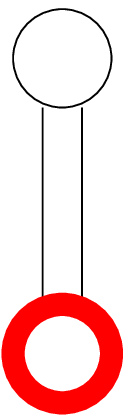}}
\put(120,93){\includegraphics[scale=0.3]{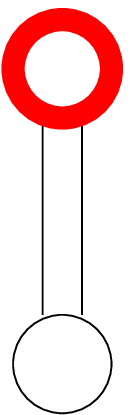}}
\put(180,93){\includegraphics[scale=0.3]{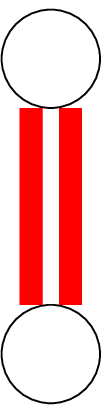}}
\put(0,13){\includegraphics[scale=0.3]{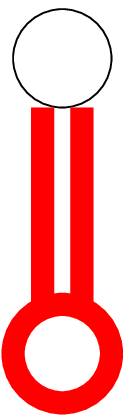}}
\put(60,13){\includegraphics[scale=0.3]{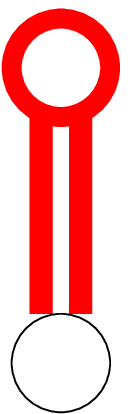}}
\put(120,13){\includegraphics[scale=0.3]{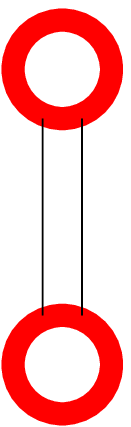}}
\put(180,13){\includegraphics[scale=0.3]{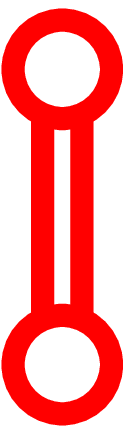}}
\put(2,80){$\emptyset$}
\put(57,80){$\{\alpha\}$}
\put(117,80){$\{\beta\}$}
\put(166,80){$\{(\alpha,\beta)\}$}
\put(-20,0){$\{\alpha,(\alpha,\beta)\}$}
\put(40,0){$\{(\alpha,\beta),\beta\}$}
\put(112,0){$\{\alpha,\beta\}$}
\put(155,0){$\{\alpha,(\alpha,\beta),\beta\}$}
\end{picture}
\]
\caption{The set ${\rm Ch}({\mathcal D})$
of chains in ${\mathcal D}$.\label{chains.fig}}
\end{figure}

We have
\begin{eqnarray*}
H_A(r)
&=&
1+\frac{r_{\alpha}}{1-r_{\alpha}}+
 \frac{r_{\beta}}{1-r_{\beta}}+\sqrt{r_{\alpha}r_{\beta}}
 +\frac{\sqrt{r_{\alpha}^3r_{\beta}}}{1-r_{\alpha}}\\
&+&
 \frac{\sqrt{r_{\alpha}r_{\beta}^3}}{(1-r_{\beta})}
 +\frac{r_{\alpha}r_{\beta}}{(1-r_{\alpha})(1-r_{\beta})}
 +\frac{\sqrt{r_{\alpha}^3r_{\beta}^3}}
 {(1-r_{\alpha})(1-r_{\beta})}\,.
\end{eqnarray*}
Setting $r=r_{\alpha}=r_{\beta}$, we obtain
the Hilbert series with respect to the usual
$\Z$-grading of $\C[{\mathcal D}]$.
\begin{eqnarray*}
h_A(r)
&=&
1+\frac{2r}{1-r}+r
 +\frac{2r^2}{1-r}
 +\frac{r^3+r^2}{(1-r)^2}\\
&=&
\frac{r+1}{(1-r)^2}\\
&=&
1+\sum_{i=1}^{\infty}(2i{+}1)r^i\,.
\end{eqnarray*}
We see that the Hilbert polynomial
is $p(i)=2i{+}1$, so
${\rm dim}({\rm Proj}~A)=1$, and
${\rm deg}({\rm Proj}~A)=2$.
\end{ex}

\begin{rem}
The Lagrangian Grassmannian $\LG(2)$ is an algebra
with straightening law on the five-element doset
obtained by adding two elements $\hat{0}<\alpha$
and $\hat{1}>\beta$ to the doset of
Example~\ref{barbell.ex}.  A similar
computation shows that
the degree of $\LG(2)$ is also $2$.
Theorem~\ref{asl.thm} allows us to
carry out such degree computations for the
Drinfel'd Lagrangian Grassmannian,
giving a new derivation of the intersection
numbers computed in quantum cohomology.
\end{rem}

Fix a poset ${\mathcal P}$, and a doset ${\mathcal D}$
on ${\mathcal P}$.  For the remainder of
this section, set $P:={\rm rank}~{\mathcal P}$
and $D:={\rm rank}~{\mathcal D}$.
Given a chain
$\{\alpha_1,\dotsc,\alpha_u,(\beta_{11},\beta_{12}),
 \dotsc,(\beta_{v1},\beta_{v2})\}\subseteq{\mathcal D}$
(not necessarily written in order),
let $r_i$ be the formal variable
corresponding to $\alpha_i$ ($i=1,\dotsb,u$),
and let $s_{jk}$ correspond to
$\beta_{jk}$ ($j=1,\dotsc,v$, $k=1,2$).
The variables $r$ and $s$
are not necessarily disjoint;
in the example above, the chain
$\{\alpha,(\alpha,\beta)\}$ has
$r_1=s_{11}$.  We have
\begin{eqnarray*}
\sum_{\stackrel{a\in{\rm im}({\bf w})}{{\rm supp}(a)=c}}r^a
&=&
\prod_{i=1}^u\frac{r_i}{1-r_i}\cdot\prod_{j=1}^v\sqrt{s_{j1}s_{j2}}
\end{eqnarray*}
Recall that we may identify ${\mathcal P}$
with the diagonal
$\Delta_{{\mathcal P}}\subseteq{\mathcal D}
 \subseteq{\mathcal P}\times{\mathcal P}$.
Letting $c_u^v$ denote the number
of chains consisting of
$u$ elements of ${\mathcal P}$
and
$v$ elements of ${\mathcal D}\setminus{\mathcal P}$,
we have
\begin{eqnarray*}
{\rm HS}_A(r)
&=&
\sum_{u=0}^{P{+}1}
 \sum_{v=0}^{D{-}P}
 c_u^v\frac{r^{u{+}v}}{(1-r)^u}\\
&=&
\sum_{u=0}^{P{+}1}
 \sum_{v=0}^{D{-}P}c_u^v
 r^{u{+}v}(\sum_{k=0}^{\infty}r^k)^u\\
&=&
\sum_{v=0}^{D{-}P}
 c_0^vr^v+\sum_{\ell=0}^{\infty}\sum_{u=1}^{P{+}1}
 \sum_{v=0}^{D{-}P}
 c_u^v\binom{u{+}\ell{-}1}{u{-}1}r^{u{+}v{+}\ell}\,.
\end{eqnarray*}
When $w>D{-}P$,
the coefficient of $r^w$ agrees
with the Hilbert polynomial:
\begin{eqnarray}\label{hilbpoly.eqn}
{\rm HP}_A(w) & = & \sum_{u=1}^{P{+}1}
              \sum_{v=0}^{D{-}P}
              c_u^v\binom{w{-}v{-}1}{u{-}1}\,.
\end{eqnarray}
In particular, the dimension of ${\rm Proj}~A$
is $P$, since
this is the largest value of
$u-1={\rm deg}_w\binom{w{-}v{-}1}{u{-}1}$.
The leading monomial of $HP_A(w)$ is
\[
\sum_{v=0}^{D{-}P}
 c_{P{+}1}^v
 \binom{w{-}v{-}1}{P{-}1}\,.
\]
By our assumption that the maximal chains in ${\mathcal P}$
(respectively, ${\mathcal D}$) have the same length,
we have $c_{P{+}1}^v=
 \binom{D{-}P}{v}
 c_{P{+}1}^0$,
so that the leading coefficient of $HP_A(w)$ is 
\begin{eqnarray*}
\frac{c_{P{+}1}^0}{(P{-}1)!}
 \sum_{v=0}^{D{-}P}
 \binom{D{-}P}{v}
&=&
\frac{2^{D{-}P}
 c_{P{+}1}^0}{(P{-}1)!},
\end{eqnarray*}
from which we deduce the degree
and dimension of ${\rm Proj}~A$.

\begin{thm}\label{degree.thm}
The degree of ${\rm Proj}~A$ is $2^{D{-}P}c_{P{+}1}^0$.
The dimension of ${\rm Proj}~A$ is $P$.
\end{thm}

\begin{ex}\label{diamond.ex}
Let $A$ be an algebra with straightening law on the doset
${\mathcal D}$ shown in Figure~\ref{diamond.fig}.

\begin{figure}[htb]
\[
\begin{picture}(50,70)
\put(0,-10){\includegraphics[scale=0.64]{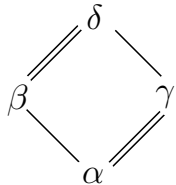}}
\end{picture}
\]
\caption{A doset on a four-element poset.\label{diamond.fig}}
\end{figure}

We have
${\rm rank}~{\mathcal P}=2$, ${\rm rank}~{\mathcal D}=3$,
and
\begin{eqnarray*}
{\rm Ch}({\mathcal D})&=&
\Bigl\{\emptyset,\{\alpha\},\{\beta\},\{\gamma\},\{\delta\},
       \{(\alpha,\gamma)\},\{(\beta,\delta)\},\\
&&\{\alpha,\gamma\},\{\alpha,\beta\},\{\alpha,\delta\},
  \{\beta,\delta\},\{\gamma,\delta\},\{\alpha,(\alpha,\gamma)\},\\
&&\{\alpha,(\beta,\delta)\},\{(\alpha,\gamma),\gamma\},
  \{(\alpha,\gamma),\delta\},\{\beta,(\beta,\delta)\},\\
&&\{(\beta,\delta),\delta\},\{\alpha,\beta,\delta\},
  \{\alpha,\gamma,\delta\},
  \{\alpha,(\alpha,\gamma),\gamma\},\\
&&\{\alpha,(\alpha,\gamma),\delta\},\{\alpha,\beta,(\beta,\delta)\},
  \{\alpha,(\beta,\delta),\delta\},\\
&&\{\beta,(\beta,\delta),\delta\},\{\alpha,(\alpha,\gamma),
  \gamma,\delta\},
  \{\alpha,\beta,(\beta,\delta),\delta\}\Bigr\}
\end{eqnarray*}
and the values of $c_u^v$ are given by the following matrix,
whose entry in row $i$ and column $j$ is $c_{j{-}1}^{i{-}1}$:
\[
\left(
\begin{array}{cccc}
 1  &  4  &  5  &  2  \\
 2  &  6  &  5  &  2 
\end{array}
\right)
\]
In view of~(\ref{hilbpoly.eqn}), the Hilbert
polynomial is therefore:
\begin{eqnarray*}
{\rm HP}_A(w)
& = &
4\binom{w-1}{0}+5\binom{w-1}{1}+2\binom{w-1}{2}\\
& + &
6\binom{w-2}{0}+5\binom{w-2}{1}+2\binom{w-2}{2}\\[5pt]
& = & 2w^2+2w+3 = 4\frac{w^2}{2!}+2w+3.
\end{eqnarray*}
In particular, ${\rm dim}({\rm Proj}~A)=2$
and ${\rm deg}({\rm Proj}~A)=4$.
\end{ex}

Theorems~\ref{asl.thm} and~\ref{degree.thm} will
allow us to compute intersection numbers in quantum
cohomology in the same
manner as Example~\ref{diamond.ex}.
The essential step is to show that the Drinfel'd
Lagrangian Grassmannian is a
algebra with straightening law
on the doset of admissible pairs ${\mathcal D_{d,n}}$.

\subsection{The doset of admissible pairs}
\label{admpairs.ssec}
We define the doset of admissible pairs on
the poset ${\mathcal P}_{d,n}$.
Let us first consider an example.
\begin{ex}
Consider the poset
\[
{\mathcal P}_{2,4}:=
\bigl\{\alpha^{(a)}\in\tangbin{4}{4}_2\mid
i\in\alpha\iff\bari\not\in\alpha\bigr\}
\]
of admissible elements of $\angbin{4}{4}_2$.
Let
${\mathcal D}_{2,4}$
be the set of elements
$(\alpha,\beta)^{(a)}\in{\mathcal O}_{{\mathcal P}_{2,4}}$
such that
$\alpha$ and $\beta$ have the same number of negative elements.
It is a doset on ${\mathcal P}_{2,4}$.
The Hasse diagram (drawn so that going up in the doset
corresponds to moving to the right) for ${\mathcal D}_{2,4}$
is shown in Figure~\ref{D24.fig}.
\begin{figure}[htb]
\[
\begin{picture}(500,120)(0,55)
\put(0,50){\includegraphics[scale=0.7]{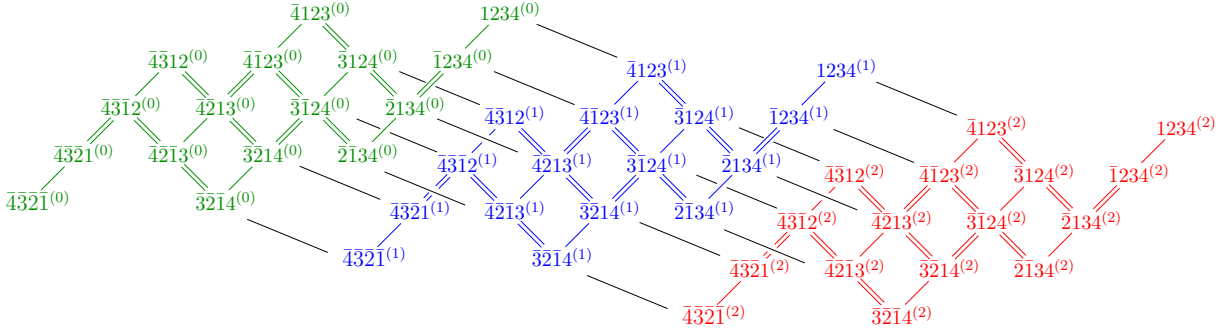}}
\end{picture}
\]
\caption{The doset ${\mathcal D}_{2,4}$.  Elements increase as
one moves to the right.\label{D24.fig}}
\end{figure}

To each
$(\alpha,\beta)^{(a)}\in{\mathcal P}_{2,4}$,
we have the Pl\"{u}cker
coordinate
$p_{(\alpha,\beta)}^{(a)}:=u^av^{d{-}a}\otimes p_{(\alpha,\beta)}
\in S^d\C^2\otimes {\rm L}(\omega_n)^*$, where
 $\{u,v\}\subseteq\C^2$ is a basis dual to $\{s,t\}\subseteq(\C^2)^*$.
\end{ex}
Let $\angbin{n}{n}_d\cong\binom{[2n]}{n}_d$
be the poset associated to the (ordinary)
Drinfel'd Grassmannian $\Q_d(n,2n)$,
and recall that ${\mathcal P}_{d,n}\subseteq\angbin{n}{n}_d$
is the subposet consisting of the elements
$\alpha^{(a)}$ such that $\alpha^t=\alpha$.  There are three types
of covers in ${\mathcal P}_{d,n}$.

\begin{enumerate}
\item $\alpha^{(a)}\lessdot\beta^{(a)}$, where $\alpha$ and $\beta$
      have the same number of negative elements.
      For example,
      $\bar4\bar213^{(a)}\lessdot\bar4\bar123^{(a)}\in {\mathcal P}_{d,4}$
      for any non-negative integers $a\leq d$.
\item $\alpha^{(a)}\lessdot\beta^{(a)}$, where the number of
      negative elements
      in $\beta$ is one less than the number of negative
      elements of $\alpha$.  For example,
      $\bar4\bar123^{(a)}\lessdot\bar4123^{(a)}\in{\mathcal P}_{d,4}$
      for any non-negative integers $a\leq d$.
\item $\alpha^{(a)}\lessdot\beta^{(a{+}1)}$, where
      the number of negative elements of $\beta$ 
      is one more than the number of negative elements
      of $\alpha$, $\bar{n}\in\beta$, and $n\in\alpha$.
      For example,
      $\bar3\bar214^{(a)}\lessdot\bar4\bar3\bar21^{(a{+}1)}$
      for any non-negative integers $a\leq d$.
\end{enumerate}

The first two types are those appearing in
the classical Bruhat order on ${\mathcal P}_{0,n}$.
It follows that ${\mathcal P}_{d,n}$ is a union of
levels ${\mathcal P}^{(a)}_{d,n}$,
each isomorphic to the Bruhat order,
with order relations between levels imposed by covers of the
type (3) above.  We define the doset
${\mathcal D}_{d,n}$ of admissible
pairs in ${\mathcal P}_{d,n}$.

\begin{defn}\label{adm.def}
A pair $(\alpha^{(a)}<\beta^{(a)})$ is {\it admissible} if
there exists a saturated chain
$\alpha=\alpha_0\lessdot\alpha_1\lessdot\dotsb\lessdot\alpha_s=\beta$,
where each $\alpha_i\lessdot\alpha_{i{+}1}$ is a cover
of type (1).
\end{defn}

We denote the set of admissible
pairs by ${\mathcal D}_{d,n}$.
Observe that the pair $(\alpha^{(a)}<\beta^{(b)})$
is never admissible if $a<b$.

\begin{prop}\label{distr.prop}
The set ${\mathcal D}_{d,n}\subseteq
 {\mathcal P}_{d,n}\times{\mathcal P}_{d,n}$
is a doset on ${\mathcal P}_{d,n}$.
The poset ${\mathcal P}_{d,n}$ is a distributive lattice.
\end{prop}

\begin{proof}
In view of our description of the covers in
${\mathcal D}_{d,n}$, it is clear that, for
all $d\geq 0$, is a doset if and only if
${\mathcal D}_n={\mathcal D}_{0,n}$ is a
doset.  The latter is proved in~\cite{DeLa79}.
To prove that ${{\mathcal P}}_{d,n}$ is a distributive lattice,
we give an isomorphism with a certain lattice
of subsets of the union of
$d{+}1$ shifted $n\times n$ squares in $\Z^2$
which generalize the usual notion of a partition.  Let
\begin{eqnarray*}
S_{d,n} &:=&
 \bigcup_{a=0}^d\{(i{+}a,j{+}a)\mid 0\leq i,j\leq n\}\,.
\end{eqnarray*}

To $\alpha^{(a)}\in{\mathcal P}_{d,n}$,
we associate the subset
of $S_{d,n}$ obtained by shifting
the (open) squares in $\alpha$ by $(a,a)$, and
adding the boxes obtained by translating a box
of $\alpha$ by a vector $(v_1,v_2)$ with
$v_1,v_2\leq 0$ and the points
$(i,i)$ for $i=0,\dotsc,a$.
See Figure~\ref{shift.fig} for an example.
\begin{figure}[htb]
\[
\begin{picture}(90,70)
\put(0,0){\includegraphics[scale=0.40]{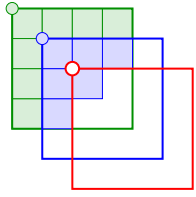}}
\end{picture}
\]
\caption{The subset of $S_{2,4}$ associated to
$\bar4\bar213^{(1)}\in{\mathcal P}_{2,4}$.\label{shift.fig}}
\end{figure}
It is straightforward to check that the (symmetric)
subsets obtained in this way form a distributive lattice
(ordered by inclusion) isomorphic to ${\mathcal P}_{d,n}$.
\end{proof}

\section{Schubert varieties and Gromov-Witten invariants}\label{enumgeom.sec}

Gromov-Witten invariants are solutions to
enumerative questions involving algebraic maps
from $\P^1$ to a projective variety $X$.  When
$X$ is the Lagrangian Grassmannian (or the
ordinary Grassmannian~\cite{SoSt01}), these
questions can be studied geometrically via
the Drinfel'd compactification, as advocated
by Braverman~\cite{Bra06}.  We do this in
Section~\ref{intprobs.ssec}, and relate our findings to
the quantum cohomology of the Lagrangian Grassmannian
in Section~\ref{qcoh.ssec}.  See~\cite{Bra06,So00c,Sot01}
for further reading on applications of Drinfel'd
compactifications to quantum cohomology.  The study
of Gromov-Witten invariants (in various special
cases) has also been approached
via the quot scheme~\cite{Ber97,Che03,Cio99,FuPa95} and
the space of stable maps~\cite{BCK05,Giv96,Opr06}.

\subsection{Intersection problems on the Drinfel'd compactification}
\label{intprobs.ssec}

Given an isotropic flag $\Fdot$ and a
symmetric partition $\alpha\in{\mathcal P}_n$,
we have the Schubert variety
\[
X^{\alpha}(\Fdot):=
\{E\in\LG(n)\mid \dim(E\cap F_{n{-}\alpha_i{+}i})\geq i\}.
\]
The enumerative problems we consider involve
conditions that the image of a map $M\in\LM_d(X)$
pass through Schubert varieties
at prescribed points of $\P^1$.
\begin{ques}\label{intprob.ques}
Let $\Fdot^1,\dotsc,\Fdot^N$ be general Lagrangian flags,
$\alpha_1,\dotsc,\alpha_N\in{\mathcal P}_n$, and
let $s_1,\dotsc,s_N\in\P^1$ be distinct points.
Assume $\sum_{i=1}^N|\alpha_i|=\dim{\LG(n)}{+}d(n{+}1)$.
How many degree-$d$ algebraic maps
$M:\P^1\to\LG(n)$ satisfy
\[
 M(s_i)\in X^{\alpha_i}(\Fdot^i)
\]
for all $i=1,\dotsc,N$?
\end{ques}
Our answer to Question~\ref{intprob.ques} is given
in Theorem~\ref{intprob.thm}.
In order to prove this result, we
must first establish some results on the
geometry of certain subvarieties of $\LQ_d(n)$
defined in terms of the {\it universal evaluation map}
${\rm ev}:\P^1\times\LM_d(n)\to\LG(n)$, defined by
${\rm ev}(s,M):=M(s)$ for $s\in\P^1$ and $M\in\LM_d(n)$.

Fix a point $s\in\P^1$ and define
${\rm ev}_s:={\rm ev}(s,\bullet):\LM_d(n)\to\LG(n)$.
Given a Schubert variety $X^{\alpha}(\Fdot)\subseteq\LG(n)$,
the set of maps $M\in\LM_d(n)$ such that $M(s)$ lies
in $X^{\alpha}(\Fdot)$ is the preimage
${\rm ev}_s^{-1}(X^{\alpha}(\Fdot))$.
This is a general translate of the locally closed subset
$X^{\alpha^{(0)}}\cap\LM_d(n)$ under the
action of the group $\SL_2\C\times\Sp_{2n}\C$.  By
a Schubert variety, we will mean the closure
of ${\rm ev}_s^{-1}(X^{\alpha}(\Fdot))$
in $\LQ_d(n)$, and denote it
by $X^{\alpha^{(0)}}(s;\Fdot)$.
In order to understand these subvarieties, we
extend the evaulation map to a globally-defined map
$\P^1\times\LQ_d(n)\to\LG(n)$. To do this, we
must first study the boundary
$\LQ_d(n)\setminus\LM_d(n)$.

The embedding $\LM_d(n)\hookrightarrow
\P((S^d\C^2)^*\otimes {\rm L}(\omega_n))$
is defined by regarding a map
$M\in\LM_d(n)$ as a $\binom{2n}{n}$-tuple
of degree-$d$ homogeneous forms. We identify the space
$\LM_d(n)$ of maps with its image, which is a locally
closed subset of $\P((S^d\C^2)^*\otimes {\rm L}(\omega_n))$.
The Drinfel'd compactification $\LQ_d(n)$ is by
definition the closure of the image.

On the other hand, $\LM_d(n)\subseteq\LQ_d(n)$ is the
set of points corresponding to a $\binom{2n}{n}$-tuple
of homogeneous forms satifying the Zariski open
condition that they have no common factor.
Therefore, the boundary $\LQ_d(n)\setminus\LM_d(n)$
consists of $\binom{2n}{n}$-tuples of homogeneous
forms which do have a common factor.  Such a list
of forms gives a regular
map of degree $a<d$ together with an effective Weil
divisor of degree $d{-}a$ on
$\P^1$ defined by the base points of the map.
We thus have a stratification
\[
 \LQ_d(n)=\bigsqcup_{a=0}^{d}\strtm{a}{d}{n}
\]
where $\P({S^a\C^2}^*)$ is the space of degree-$a$
forms in two variables, or alternatively,
the space of effective Weil divisors on $\P^1$
of degree $a$.    In particular,
the boundary of $\LQ_d(n)$ is simply
$\bigsqcup_{a=1}^{d}\strtm{a}{d}{n}$.
We may regard any
point of $\LQ_d(n)$ as
a pair $(D,M)$, where $M\in\LM_{d{-}a}(n)$ and
$D$ is a divisor on $\P^1$.

Fixing a point $s\in\P^1$, the
evaluation map ${\rm ev}_s:={\rm ev}(s,\bullet)$ is
undefined at each point
$(D,M)\in\strtm{a}{d}{n}$
such that $s\in D$.  Thus, restricting to the
stratum $\strtm{a}{d}{n}$, the map
${\rm ev}_s$ is defined on $U_s^a\times\LM_{d{-}a}(n)$,
where $U_s^a\subseteq\P(S_a\C^2)$ is the set of
forms which do not vanish at $s\in\P^1$.

For each $a=0,\dotsc,d$, define a map
$\epsilon_s^a:\strtm{a}{d}{n}\to\LG(n)$
by the formula $\epsilon_s^a(D,M):=M(s)$,
and let
$\epsilon_s:\LQ_d(n)=\sqcup_a\strtm{a}{d}{n}\to\LG(n)$
be the (globally-defined) map which restricts to
$\epsilon_s^a$ on $\strtm{a}{d}{n}$.  The
evaluation map ${\rm ev}_s$ agrees with $\epsilon_s$
wherever it is defined.  Hence $\epsilon_s$ extends
${\rm ev}_s$ to a globally defined map, which is a
morphism on each stratum $\strtm{a}{d}{n}$.  The
Schubert variety $X^{\alpha^{(0)}}(s;\Fdot)$ is the
preimage of $X^{\alpha}(\Fdot)$ under this globally
defined map; hence we have the following fact.
\begin{lemma}\label{stratum.lem}
Given a point $s\in\P^1$ and a isotropic flag $\Fdot$,
the Schubert variety $X^{\alpha^{(0)}}(s,\Fdot)$ is the disjoint
union of the strata $\P({S^a\C^2}^*)\times
 (X^{\alpha^{(a)}}(s;\Fdot)\cap\LM_{d{-}a}(n))$.
\end{lemma}
\begin{proof}
For each $a\in\{0,\dotsc d\}$, we have
$
\left(\strtm{a}{d}{n}\right)~\cap~\epsilon_s^{-1}(X^{\alpha}(\Fdot))
=\P({S^a\C^2}^*)\times X^{\alpha^{(a)}}(s;\Fdot).
$
\end{proof}
We now state and prove the main theorem of this section.
\begin{thm}\label{intprob.thm}
Given partitions $\alpha_1,\dotsc\alpha_N\in{\mathcal P}_n$
such that $\sum_{i=1}^N|\alpha_i|=\binom{n{+}1}{2}{+}d(n{+}1)$,
general isotropic flags $\Fdot^1,\dotsc,\Fdot^N$, and
distinct points $s_1,\dotsc,s_N\in\P^1$,
the intersection
\begin{equation}\label{schint.eqn}
 X^{\alpha_1^{(0)}}(s_1;\Fdot^1)\cap\dotsb\cap{X^{\alpha_N^{(0)}}(s_N;\Fdot^N)}
\end{equation}
is transverse, and hence consists only of reduced points.
Each point of the intersection~(\ref{schint.eqn})
lies in $\LM_d(n)$, {\it i.e.}, corresponds to
a degree-$d$ map whose image $M(s_i)$
lies in $X^{\alpha_i}(\Fdot^i)$ for $i=1,\dotsc,N$.
\end{thm}
\begin{proof}
For each $a=0,\dotsc,d$, the Schubert variety
$X^{\alpha^{(a)}}(s,\Fdot)$ is the preimage
of the Schubert variety $X^{\alpha}(\Fdot)\subseteq\LG(n)$
under the evaluation map $\epsilon_s$, which is regular
on the stratum $\strtm{a}{d}{n}$.  By Lemma~\ref{stratum.lem},
it suffices to consider the intersection~(\ref{schint.eqn})
on each of these strata.  Fix $a\in\{0,\dotsc,d\}$, and
consider the product of evaluation maps
\[
\prod_{i=1}^N\epsilon_{s_i}:\left(\strtm{a}{d}{n}\right)^N\to\LG(n)^N
\]
and the injection
\[
X^{\alpha_1}(\Fdot^1)\times\dotsb\times{X^{\alpha_n}(\Fdot^N)}
 \hookrightarrow\LG(n)^N.
\]
The intersection
$(X^{\alpha_1^{(0)}}\cap\dotsb\cap{X^{\alpha_N^{(0)}}})
\cap\strtm{a}{d}{n}$
is isomorphic to the fiber product
\[
\left(\strtm{a}{d}{n}\right)^N\times_{\LG(n)^N}
 \left(X^{\alpha_1}(\Fdot^1)\times\dotsb\times{X^{\alpha_N}(\Fdot^N)}\right).
\]
For each $a=0,\dotsc,d$, Kleiman's theorem~\cite[Corollary 2]{Kle74}
implies that this intersection is proper and tranverse.
Considering the dimensions of these subvarieties,
we see that this intersection is therefore zero-dimensional
when $a=0$ and empty when $a>0$.
\end{proof}
%
\subsection{Gromov-Witten invariants and Quantum Cohomology}
\label{qcoh.ssec}

A common approach to Question~\ref{intprob.ques} is through the
quantum cohomology ring of the Lagrangian Grassmannian
$QH^*(\LG(n))$, defined as follows.  The cohomology
ring $H^*(\LG(n);\Z)$ has a $\Z$-basis consisting
of the classes of the Schubert
variety (the {\it Schubert classes})
$\sigma_{\alpha}:=[X^{\alpha}]$, where
$\alpha\in{\mathcal P}_n$.
We will denote by $\alpha^*$
the dual partition, defined so that
$\sigma_{\alpha}\cdot\sigma_{\alpha^*}
=1\in H^*(\LG(n);\Z)$~\cite{HiBo86}.
The correspondence $\alpha\leftrightarrow\alpha^*$
is bijective and order reversing.

The (small) {\it quantum cohomology ring} is
the $\Z[q]$-algebra isomorphic to
$H^*(\LG(n);\Z)\otimes\Z[q]$ as a $\Z[q]$-module, and
with multiplication defined by the formula
\begin{eqnarray}\label{qmult.eqn}
 \sigma_{\alpha}\cdot\sigma_{\beta}
 &=&
 \sum\langle\alpha, \beta, \gamma^*\rangle_d~\sigma_{\gamma}q^d,
\end{eqnarray}
where the sum is over all $d\geq 0$ and $\gamma$ such that
$|\gamma|=|\alpha|{+}|\beta|{-}d\binom{n{+}1}{2}$.  For partitions
$\alpha$, $\beta$, and $\gamma$ in ${\mathcal P}_n$, the coefficients
$\langle\alpha, \beta, \gamma^*\rangle_d$ are the
{\it Gromov-Witten invariants}, defined as the
number of algebraic maps $M:\P^1\to\LG(n)$ of degree $d$ such
that $M(0)\in X^{\alpha}(F_{\bullet})$, $M(1)\in X^{\beta}(G_{\bullet})$,
and $M(\infty)\in X^{\gamma^*}(H_{\bullet})$, where $F_{\bullet}$,
$G_{\bullet}$, and $H_{\bullet}$ are general isotropic flags
({\it i.e.}, general translates of the standard
flag under the action of the group $\Sp_{2n}\C$).

A special case of Pieri's rule gives a formula for the product
of a Schubert class $\sigma_{\alpha}\in H^*(\LG(n);\Z)$ with the
{\it simple} Schubert class $\sigma_{_{\one}}$~\cite{HiBo86}:
\[
 \sigma_{\alpha}\cdot\sigma_{_{\one}}
 =\sum_{\alpha\lessdot\beta}2^{N(\alpha,\beta)}\sigma_{\beta},
\]
the sum over all partitions $\beta$ obtained from
$\alpha$ by adding a box above the diagonal,
along with its image under reflection about the
diagonal.  The exponent $N(\alpha,\beta)=1$
if $(\alpha,\beta)\in{\mathcal D}_n$
and $N(\alpha,\beta)=0$ otherwise
({\it cf}. Proposition~\ref{fiber.prop}).
In~\cite{KrTa03}, Kresch and Tamvakis give a
quantum analogue of Pieri's rule.  We state the 
relevant special case of this rule:
\begin{prop}\label{qpr.prop}\cite{KrTa03}
 For any $\alpha\in{\mathcal P}_n$, we have
 \[
  \sigma_{\alpha}\cdot\sigma_{_{\one}}
  =\sum_{\alpha\lessdot\beta}2^{N(\alpha,\beta)}\sigma_{\beta}
  +\sigma_{\gamma}q
 \]
 in $QH^*(\LG(n))$, where the first sum is from the classical Pieri
 rule, and $\sigma_{\gamma}=0$ unless $\alpha$ contains the
 hook-shaped partition $(n,1^{n{-}1})$, in which case
 $\gamma$ is the partition obtained from
 $\alpha$ by removing this hook.
\end{prop}
For $\alpha^{(a)}\lessdot\beta^{(b)}$
in ${\mathcal D}_{d,n}$, let
$N'(\alpha^{(a)},\beta^{(b)})=N(\alpha,\beta)$ if
$b=a$, and let $N'(\alpha^{(a)},\beta^{(b)})=0$
if $b=a+1$.
Let $\alpha^{(0)}\in{\mathcal P}_{d,n}$, and let $\pi\in\N$
be its {\it corank} in ${\mathcal P}_{d,n}$;
that is, $\pi$ is the length of any
saturated chain of elements
$\alpha^{(d)}=x_0\lessdot\dotsb\lessdot x_{\pi}=(n^n)^{(d)}$,
where $x_i\in{\mathcal P}_{d,n}$ for
all $i=0,\dotsc,\pi$ and $(n^n)^{(d)}$ is the maximal
element of ${\mathcal P}_{d,n}$.  By Theorem~\ref{degree.thm},
$\pi$ is the dimension of $X^{\alpha^{(0)}}$.
The quantum Pieri rule of Proposition~\ref{qpr.prop}
has a simple formulation in terms of the distributive
lattice ${\mathcal P}_{d,n}$:

\begin{thm}\label{coh.prod.thm}
The quantum Pieri rule~\ref{qpr.prop} has the following
formulation in terms of the poset ${\mathcal P}_{d,n}$:
\[
 (\sigma_{\alpha}q^a)\cdot\sigma_{_{\one}}
 =\sum_{\alpha^{(a)}\lessdot\beta^{(b)}}2^{N'(\alpha^{(a)},\beta^{(b)})}
   \sigma_{\beta}q^b
\]
As a consequence, we have
\begin{eqnarray*}
 \sigma_{\alpha}\cdot(\sigma_{_{\one}})^{\pi}
 &=&\deg(X^{\alpha^{(0)}})\cdot\sigma_{(n^n)}q^d~({\rm mod}~d{+}1).
\end{eqnarray*}
\end{thm}
\begin{proof}
The element $\gamma^{(a{+}1)}\in{\mathcal P}_{d,n}$,
where $\gamma$ is as the partition obtained by removing
a maximal hook from $\alpha$ in Proposition~\ref{qpr.prop},
is the unique cover of $\alpha^{(a)}\in{\mathcal P}_{d,n}$ with
superscript $a{+}1$.  The remaining covers (with
superscript $a$) index the sum in Proposition~\ref{qpr.prop}.

The second formula follows by induction from the first.
\end{proof}

The appearance of the number $\deg(X^{\alpha^{(0)}})$
in the last statement of Theorem~\ref{coh.prod.thm} is for purely
combinatorial reasons: it is the number of saturated
chains $\alpha^{(0)}\lessdot\dotsb\lessdot{(n^n)^{(d)}}$
in ${\mathcal D}_{d,n}$, counted with multiplicity.
Since $X^{\one^{(0)}}$ is a hyperplane section of $\LQ_d(n)$,
this is also the number of points in the intersection
\begin{equation}\label{qpi.eqn}
 X^{\alpha^{(0)}}(s;\Fdot)\cap\left(
 \bigcap_{i=1}^{\pi}{X^{\one^{(0)}}(s_i;\Fdot^i)}\right),
\end{equation}
the intersection of $X^{\alpha^{(0)}}(s;\Fdot)$ with
$\pi={\rm codim}(X^{\alpha^{(0)}}(s;\Fdot))$ general translates
of the hyperplane section $X^{\one^{(0)}}$.
On the other hand, multiplication
in $QH^*(\LG(n))$ represents the conjunction of conditions
that a map takes values in Schubert varieties at generic
points of $\P^1$.  In this way, the quantum cohomology
identity of Theorem~\ref{coh.prod.thm} has an interpretation as
the number of points in the intersection~(\ref{qpi.eqn})
of Schubert varieties in $\LQ_d(n)$.

\section{The straightening law}\label{str8law.sec}

\subsection{A basis for $S^d\C^2\otimes {\rm L}(\omega_n)^*$}

The Drinfel'd Lagrangian
Grassmannian embeds in the projective space
$\P((S^d\C^2)^*\otimes {\rm L}(\omega_n))$.
We begin by describing convenient bases
for the representation ${\rm L}(\omega_n)$
and its dual ${\rm L}(\omega_n)^*$.

For $\alpha\in\angbin{n}{n}$ and
positive integers, set
$v_{\alpha}^{(a)}:=s^at^{d{-}a}\otimes
 e_{\alpha_1}\wedge\dotsb\wedge e_{\alpha_n}
 \in (S^d\C^2)^*\otimes\bigwedge^n\C^{2n}$,
and let
$p_{\alpha}^{(a)}:=u^av^{d{-}a}\otimes
 e_{\alpha_1}^*\wedge\dotsb\wedge e_{\alpha_n}^*
 \in S^d\C^2\otimes\bigwedge^n{\C^{2n}}^*$
be the {\it Pl\"ucker coordinate}
indexed by $\alpha^{(a)}\in{\mathcal D}_{d,n}$,
where $\{u,v\}\in\C^2$ and
$\{s,t\}\in(\C^2)^*$ are dual bases.

The representation
${\rm L}(\omega_n)^*$ is the quotient of
$\bigwedge^n{\C^{2n}}^*$
by the linear subspace
$L_n=\Omega\wedge\bigwedge^{n{-}2}{\C^{2n}}^*$
described in Proposition~\ref{ctxn.prop}.  Thus
$S^d\C^2\otimes {\rm L}(\omega_n)^*$ is
the quotient of
$S^d\C^2\otimes \bigwedge^n{\C^{2n}}^*$
by the linear subspace:
\begin{eqnarray*}
L_{d,n} &:=& 
 S^d\C^2\otimes L_n\,.
\end{eqnarray*}
Note that $L_{d,n}$ is spanned by the linear forms
\begin{eqnarray}\label{ldn.eqn}
\ell_{\alpha}^{(a)} &:=& u^av^{d{-}a}\otimes
 \sum_{i\mid\{\bari,i\}\cap\alpha=\emptyset}
  e_{\bari}^*\wedge e_i^*\wedge
  e_{\alpha_1}^*\wedge\dotsb\wedge e_{\alpha_{n{-}2}}^*
\end{eqnarray}
for $\alpha\in\angbin{n}{n-2}$
and $a=0,\dotsc,d$.  The linear form
(\ref{ldn.eqn}) is simply $u^av^{d{-}a}$
tensored with a linear form generating
$L_n$.  Each term in the
linear form~(\ref{ldn.eqn}) is a Pl\"ucker
coordinate indexed by a sequence
involving both $i$ and $\bari$, for some $i\in[n]$.

Let $S\subseteq\SL_2(\C)$ and
$T\subseteq\Sp_{2n}(\C)$ be
maximal tori.  The torus $S$ is one
dimensional, so that its
Lie algebra ${\mathfrak s}$ has
basis consisting of a single
element $H\in{\mathfrak s}$.
For $i\in\angles{n}$, let
$h_i:=E_{ii}-E_{\bari\bari}$.
The set $\{h_i\mid i\in[n]\}$ is
a basis for the Lie algebra
${\mathfrak t}$ of $T\subseteq\Sp_{2n}(\C)$.
The weights of the maximal torus
$S\times T\subseteq\SL_2(\C)\times\Sp_{2n}(\C)$
are elements of ${\mathfrak s}^*\oplus{\mathfrak t}^*$. 
The Pl\"ucker coordinate
$p_{\alpha}^{(a)}\in S^d\C^2\otimes(\bigwedge^n\C^{2n})^*$
is a weight vector of weight
\begin{equation}\label{weight.eqn}
(d{-}2a)H^*
 +\sum_{i\mid\bar\alpha_i\not\in\alpha}h_{\alpha_i}^*.
\end{equation}
Each linear form~(\ref{ldn.eqn})
lies in a unique weight space.
Thus, to find a basis for
$S^d\C^2\otimes{\rm L}(\omega_n)^*$,
it suffices to find a basis for
each weight space.  We therefore fix
the weight~(\ref{weight.eqn}) and its
corresponding weight space in the
following discussion.  We reduce
to the case that the
weight~(\ref{weight.eqn}) is in
fact $0$, as follows.

For each $\alpha\in\angbin{n}{n{-}2}$,
we have an element
$\ell_{\alpha}=\Omega\wedge p_{\alpha}\in L_n$.  This is
a weight vector of weight
$\omega_{\alpha}:=
h_{\alpha_1}^*{+}\dotsb{+}h_{\alpha_k}^*\in{\mathfrak t}^*$.
Set $\widetilde{\alpha}:=\{i\in\alpha\mid\bari\not\in\alpha\}$
and observe that
$\omega_{\widetilde{\alpha}}=\omega_{\alpha}$.
The elements $\alpha\in\angbin{n}{n{-}2}$ such that
$\ell_{\alpha}\in(L_n)_{\omega}$ are those
satisfying $\omega_{\alpha}=\omega$.  That is,
$(L_n)_{\omega}=
\langle\Omega\wedge p_{\alpha}
\mid\omega_{\alpha}=\omega\rangle$.

The shape of the linear form $\ell_{\alpha}$
is determined by the number of pairs
$\{\bari,i\}\subseteq\alpha$;
it is the same,
up to multiplication of
some variables by $-1$,
as the linear form
$\ell_{\alpha\setminus\widetilde{\alpha}}=
 \Omega\wedge p_{\alpha\setminus\widetilde{\alpha}}
 \in L_{n{-}|\widetilde{\alpha}|}$,
of weight $\omega_{\alpha\setminus\widetilde{\alpha}}=0$.
It follows that the generators of
$(L_n)_{\omega_{\alpha}}$ have the
same form as those of
$(L_{n{-}|\widetilde{\alpha}|})_{\omega_{\alpha\setminus\widetilde{\alpha}}}$,
up to some signs arising from sorting
the indices.  Since these signs do not
affect linear independence, it suffices to
find a basis for $(L_n)_0$,
from which it is then straightforward
to deduce a basis for 
$(L_n)_{\omega_{\alpha}}$.
We thus assume that
the weight space in question is
$(L_n)_0$.
This implies that $n$ is even; set $m:=n/2$.

\begin{ex}
We consider linear forms which span
$(L_6)_{h_1^*+h_3^*}$.
Let $m=3$ (so $n=6$) and $\omega=h_1^*+h_3^*$.
If $\alpha=\bar6136$, then
$\widetilde{\alpha}=13$
and $\omega_{\alpha}=\omega$.  We have
$\ell_{\alpha}=p_{\bar6\bar51356}
+p_{\bar6\bar41346}-
 p_{\bar6\bar21236}$.
The equations for the weight
space $(L_n)_{\omega}$ are
\begin{eqnarray*}
\ell_{\bar6136}&=&
p_{\bar6\bar51356}
+p_{\bar6\bar41346}-p_{\bar6\bar21236}\\
\ell_{\bar5135}&=&
p_{\bar6\bar51356}
+p_{\bar5\bar41345}-p_{\bar5\bar21235}\\
\ell_{\bar4134}&=&
p_{\bar6\bar41346}
+p_{\bar5\bar41345}-p_{\bar4\bar21234}\\
\ell_{\bar2123}&=&
p_{\bar6\bar21236}
+p_{\bar5\bar21235}+p_{\bar4\bar21234}
\end{eqnarray*}
We can obtain the linear forms which span
$(L_4)_0$ (see Example~\ref{lg4})
by first removing every occurrence of $1$ and $3$ in
the subscripts above and then flattening the
remaining indices.  That is, we apply the following
replacement (and similarly for the negative indices):
$6\mapsto 4$, $5\mapsto 3$,
$4\mapsto 2$, and $2\mapsto 1$.
We then replace a variable by its negative if
$2$ appears in its index; this is to keep track of
the sign of the permutation sorting the sequence
$(\bari,i,\alpha_1,\dotsc,\alpha_{n{-}2})$ in each
term of $\ell_{\alpha}$ (see Equation~\ref{ldn.eqn}).
\end{ex}

By Proposition~\ref{ctxn.prop}, the map
\[
\Bigl(\bigwedge^{2m}\C^{4m}\Bigr)_0\longrightarrow
\Bigl(\bigwedge^{2m{-}2} \C^{4m}\Bigr)_0
\]
given by contraction with the form
$\Omega\in\bigwedge^2(\C^{4m})^*$
is surjective, with kernel
$\bigl({\rm L}(\omega_{2m})\bigr)_0$.
Since the set
$\bigl\{(\bar{\alpha},\alpha)~\bigl\lvert\bigr.~\alpha
 \in\binom{[2m]}{k}\bigr\}$
is a basis of $\bigl(\bigwedge^{2k} \C^{4m}\bigr)_0$
(for any $k\leq m$), we have
\begin{eqnarray*}
\dim\bigl({\rm L}(\omega_{2m})\bigr)_0
&=&
\dim\Bigl(\bigwedge^{2m}\C^{4m}\Bigr)_0
 -\dim\Bigl(\bigwedge^{2m{-}2}\C^{4m}\Bigr)_0\\
&=&\tbinom{2m}{m}-\tbinom{2m}{m-2}\\
&=&
\tfrac{1}{m{+}1}\tbinom{2m}{m}\,.
\end{eqnarray*}
This number is equal to the number
of admissible pairs of weight $0$.

\begin{lemma}
$\dim\bigl({\rm L}(\omega_n)\bigr)_0$
is equal to the number of admissible pairs
$(\alpha,\beta)\in{\mathcal D}_{n}$ of weight 
$\frac{\omega_{\alpha}{+}\omega_{\beta}}{2}=0$.
\end{lemma}

\begin{proof}
Recall that each trivial admissible pair
$(\alpha,\alpha)$, where
\[
\alpha=(\bar{a}_1,\dotsc,\bar{a}_s,b_1,\dotsc,b_{n-s})
 \in{\mathcal D}_{n}
\]
indexes a weight vector, of weight
$\sum_{i=1}^{n{-}s}h_{b_i}^*-\sum_{i=1}^sh_{a_i}^*$.
Also, the non-trivial admissible pairs
are those $(\alpha,\beta)$ for
which $\alpha<\beta$ have the same
number of negative elements.  Therefore,
the admissible pairs of weight zero are
the $(\alpha,\beta)\in{\mathcal D}_{n}$
such that $\beta=(\bar{a}_m,\dotsc ,\bar{a}_1,b_1,\dotsc,b_m)$,
$\alpha=(\bar{b}_m,\dotsc ,\bar{b}_1,a_1,\dotsc,a_m)$, and
the sets
$\{a_1,\dotsc,a_m\}$ and
$\{b_1,\dotsc,b_m\}$ are disjoint.
This last condition is equivalent to $a_i>b_i$
for all $i\in[m]$.  The number
of such pairs is equal to
the number of standard tableaux
of shape $(m^2)$ (that is,
a rectangular box with $2$ rows
and $m$ columns) with entries
in $[2m]$.  By the hook
length formula~\cite{Ful97}
this number is $\frac{1}{m{+}1}\binom{2m}{m}$.
\end{proof}

The weight vectors
$p_{\alpha}\in\bigl({\bigwedge^n\C^{4m}}^*\bigr)_0$
are indexed by sequences of the form
\[
\alpha=
 (\bar\alpha_m,\dotsc,\bar\alpha_1,\alpha_1,\dotsc,\alpha_m)
\]
which can be abbreviated by the positive subsequence
$\alpha_+:=(\alpha_1,\dotsc,\alpha_m)\in\binom{[2m]}{m}$
without ambiguity.
We take these as an indexing set for the variables appearing
in the linear forms~(\ref{ldn.eqn}).

With this notation, the positive parts of Northeast sequences
are characterized in Proposition~\ref{NEseq.prop}.  The
proof requires the following definition.
\begin{defn}
A {\it tableau} is a partition whose boxes
are filled with
integers from the set $[n]$, for some $n\in\N$.
A tableau is {\it standard} if the entries strictly
increase from left to right and top to bottom.
\end{defn}
\begin{prop}\label{NEseq.prop}
Let $\alpha\in\angbin{2m}{2m}$ be a Northeast sequence.
Then the positive part of $\alpha$ satisfies
$\alpha_+\geq{24\dotsb(2m)}\in\binom{[2m]}{m}$.
In particular, no Northeast sequence contains $1\in[2m]$
and every Northeast sequence contains $2m\in[2m]$.
\end{prop}
\begin{proof}
$\alpha_+\geq{24\dotsb(2m)}$ if and only if
the tableau of shape $(m^2)$ whose first
row is filled with the sequence
$(\alpha^t)_+=[n]\setminus\alpha_+$ and whose
second row is filled with the $\alpha_+$ is
standard.  This is equivalent to $\alpha$ being
Northeast.
\end{proof}

It follows from Proposition~\ref{fiber.prop}
that the set ${\mathcal{NE}}$ of
Northeast sequences indexing vectors
of weight zero
has cardinality equal to
the dimension of the
zero-weight space of the representation
${\rm L}(\omega_{2m})^*$.  This weight space is
the cokernel of the map
\[
\Omega\wedge\bullet:
  \Bigl(\bigwedge^{2m{-}2}{\C^{4m}}\Bigr)_0^*
   \longrightarrow\Bigl(\bigwedge^{2m}{\C^{4m}}\Bigr)_0^*\,\,.
\]
Similarly, the weight space ${\rm L}(\omega_{2m})_0$
is the kernel of the dual map
\[
\Omega\contr\bullet:
  \Bigl(\bigwedge^{2m}{\C^{4m}}\Bigr)_0
   \longrightarrow\Bigl(\bigwedge^{2m{-}2}{\C^{4m}}\Bigr)_0\,\,.
\]

We fix the positive integer $m$, and
consider only the positive subsequence
$\alpha_+$ of the sequence
$\alpha\in\angbin{2m}{2m}$.
When the weight of $p_{\alpha}$ is $0$,
$\alpha_+$ is an element of $\binom{[2m]}{m}$.
For $\alpha\in\binom{[2m]}{m}$, we call
a bijection $M:\alpha\rightarrow\alpha^c$
a {\it matching} of $\alpha$.
Fixing a matching $M:\alpha\to\alpha^c$,
we have an element of
the kernel ${\rm L}(\omega_{2m})$, as follows.
Let $H_{\alpha}$ be the set of
all sequences in $\binom{[2m]}{m}$
obtained by interchanging $M(\alpha_i)$
and $\alpha_i$, for $i\in I$, $I\subseteq[m]$.

Elements of the set $H_{\alpha}$ are the
vertices of an $m$-dimensional hypercube, whose
edges connect pairs of sequences
which are related by the interchange
of a single element.  Equivalently,
a pair of sequences are connected
by an edge if they share a subsequence
of size $m{-}1$.  For any such
subsequence $\beta\subseteq\alpha$
there exists a unique edge of $H_{\alpha}$
connecting the two vertices which
share the subsequence $\beta$.
Let $I\cdot\alpha$
denote the element of $H_{\alpha}$
obtained from $\alpha$ by the
interchange of $M(\alpha_i)$
and $\alpha_i$ for $i\in I$.
The element
\[
 K_{\alpha}:=\sum_{I\subseteq[m]}(-1)^{|I|}v_{I\cdot\alpha}
\]
lies in the kernel ${\rm L}(\omega_{2m})$.  Indeed,
for each $I\subseteq[m]$, we have
\[
\Omega\contr v_{I\cdot\alpha}=
 \sum_{i=1}^mv_{(I\cdot\alpha)\setminus\{(I\cdot\alpha)_i\}}.
\]
For each term 
$v_{(I\cdot\alpha)\setminus\{(I\cdot\alpha)_i\}}$
on the right-hand side, let
$j\in[m]$ be such that
either $(I\cdot\alpha)_i=\alpha_j$
or $(I\cdot\alpha)_i=\alpha_j^c$.
Set
\begin{eqnarray*}
J &:=&
\left\{\begin{array}{cc}
I\cup\{j\}, & {\text if } (I\cdot\alpha)_i=\alpha_j,\\
I\setminus\{j\}, & {\text if } (I\cdot\alpha)_i=\alpha_j^c.
\end{array}\right.
\end{eqnarray*}
The set $J$ is the
unique subset of $[m]$ such that
$(I\cdot\alpha)\setminus\{(I\cdot\alpha)_i\}$
is in the support of 
$\Omega\contr v_{J\cdot\alpha}$,
with coefficient $(-1)^{|J|}=(-1)^{|I|{+}1}$.
Hence these terms cancel
in $K_{\alpha}$, and we
see that the coefficient of each
$v_{\beta}$ ($\beta\in\binom{[2m]}{m-1}$)
in the support of
$\Omega\contr K_{\alpha}$ is zero.
Therefore $\Omega\contr K_{\alpha}=0$.
See Example~\ref{lg4.ex} for the
case $m=2$.

If $\alpha\in{\mathcal{NE}}$ then that there
exists a {\it descending matching},
that is, $M(\alpha_i)<\alpha_i$ for
all $i\in[m]$.
For example, the condition that the
matching $M(\alpha_i):=\alpha_i^c$ be
descending is equivalent to the condition
that $\alpha$ be Northeast.
If we choose a descending matching for each
$\alpha\in{\mathcal{NE}}$, the
element $K_{\alpha}\in{\rm L}(\omega_{2m})$
is supported on sequences which
precede $\alpha$ in the poset $\binom{[2m]}{m}$.
It follows that the set
${\mathcal B}:=
 \{K_{\alpha}\in{\rm L}(\omega_{2m})\mid\alpha\in{\mathcal{NE}}\}$
is a basis for ${\rm L}(\omega_{2m})$.

\begin{lemma}\label{basis.lemma}
The Pl\"ucker coordinates
$p_{\alpha}$ with $\alpha\in{\mathcal{NE}}$
are a basis for ${\rm L}(\omega_{2m})^*$.
\end{lemma}

\begin{proof}
Fix a basis ${\mathcal B}$
of ${\rm L}(\omega_{2m})$ obtained from
descending matchings of each
Northeast sequence.
We use this basis to show that
the set of Pl\"ucker coordinates
$p_{\alpha}$ such that $\alpha$
is Northeast is a basis for the
dual ${\rm L}(\omega_{2m})^*$.

Suppose not.  Then there exists
a linear form
$\ell=\sum_{\alpha\in{\mathcal{NE}}}c_{\alpha}p_{\alpha}$
vanishing on each element of
the basis ${\mathcal B}$.  We show
by induction on the poset ${\mathcal{NE}}$
that all of the coefficients
$c_{\alpha}$ appearing
in this form vanish.

Fix a Northeast sequence
$\alpha\in{\mathcal{NE}}$, and
assume that $c_{\beta}=0$
for all Northeast $\beta<\alpha$.
Since $K_{\alpha}$ involves only
the basis vectors $v_{\beta}$ with
$\beta\leq\alpha$, we have
$\ell(K_{\alpha})=c_{\alpha}$,
hence $c_{\alpha}=0$.  This
completes the inductive step
of the proof.

The initial step of the induction
is simply the inductive step applied
to the unique minimal Northeast
sequence $\alpha=24\dotsb(2m)$.
\end{proof}

It follows that every Pl\"ucker
coordinate $p_{\alpha}$ indexed
by a non-Northeast sequence
$\alpha$
can be written uniquely as a
linear combination of Pl\"ucker
coordinates indexed by Northeast
sequences.
We can be more precise about the 
form of these linear combinations.
Recall that each fiber of the map
$\pi_{2m}$ contains a unique
Northeast sequence.  For a
sequence $\alpha_0$, let
$\alpha$ be the Northeast
sequence in the same fiber
as $\alpha_0$.

\begin{lemma}\label{smaller.lemma}
For each non-Northeast sequence
$\alpha_0$, let $\ell'_{\alpha_0}$ be
the linear relation among the
Pl\"ucker coordinates expressing
$p_{\alpha_0}$ as a linear combination
of the $p_{\beta}$ with $\beta$
Northeast.  Then $p_{\alpha}$
appears in $\ell'_{\alpha_0}$ with
coefficient $(-1)^{|I|}$, where
$\alpha=I\cdot\alpha_0$, and every
other Northeast $\beta$ with
$p_{\beta}$ in the support
of $\ell'_{\alpha_0}$ satisfies
$\beta>\alpha$.
\end{lemma}

\begin{proof}
Let $M$ be the descending matching
of $\alpha$ with $\alpha^c$
defined by
$M(\alpha_i):=\alpha_i^c$.
Let $K_{\alpha}$ be the
kernel element obtained by
the process described above.
Any linear form
\[
\ell=
 p_{\alpha_0}+(-1)^{|I|{+}1}p_{\alpha}
  +\sum_{\alpha<\beta\in{\mathcal{NE}}}c_{\beta}p_{\beta}
\]
vanishes on $K_{\alpha}$.

We extend this relation to one
which vanishes on all of ${\rm L}(\omega_n)_0$,
proceeding inductively on the poset
of Northeast sequences greater than
or equal to $\alpha$.  Suppose
that $\beta>\alpha$ is Northeast.
By induction, suppose that for each
Northeast sequence
$\gamma$ in the interval
$[\alpha,\beta]$ the coefficient
$c_{\gamma}$ of $\ell$
has been determined in such a
way that $\ell(K_{\gamma})=0$.

Let $S$ be the set of Northeast
sequences $\gamma$ in the open interval
$(\alpha,\beta)$ such that
$v_{\gamma}$ appears in $K_{\beta}$.
Then
\[
\ell(K_{\beta})=\Bigl(\sum_{\gamma\in S}c_{\gamma}\Bigr)+c_{\beta},
\]
so setting $c_{\beta}:=-\sum_{\gamma\in S}c_{\gamma}$
implies that $\ell(K_{\beta})=0$.  

This completes the inductive part of the proof.
We now have a linear form
$\ell$ vanishing on ${\rm L}(\omega_n)_0$
which expresses $p_{\alpha_0}$
as a linear combination
of Pl\"ucker coordinates indexed by Northeast
sequences.  Since such a linear form is unique,
$\ell=\ell'_{\alpha_0}$.
\end{proof}

By Lemmas~\ref{basis.lemma} and~\ref{smaller.lemma}
and the argument preceding them, we deduce the following
theorem.

\begin{thm}\label{normalform.thm}
The system of linear relations
\[
 \bigl\{\ell_{\alpha}^{(a)}=u^av^{d{-}a}
  \otimes\Omega\wedge p_{\alpha}~\bigl\lvert\bigr.~
   a=0,\dotsc,d,~\alpha\in\tangbin{n}{n{-}2}\bigr\}
\]
has a reduced normal form consisting of
linear forms expressing each Pl\"ucker
coordinate $p_{\beta}^{(b)}$ with
$\beta\not\in{\mathcal{NE}}\subseteq\angbin{n}{n}$
as a linear combination of Pl\"ucker
coordinates indexed by Northeast
elements of $\angbin{n}{n}$.
\end{thm}

\begin{proof}
We have seen that the linear relations
preserve weight spaces, and
Lemmas~\ref{basis.lemma} and~\ref{smaller.lemma}
provide the required normal form on each of these.
The union of the relations constitute a normal
form for the linear relations generating the
entire linear subspace $L_{d,n}$.
\end{proof}

\begin{ex}\label{lg4.ex}
Consider the zero weight space
$\left(\bigwedge^4\C^8\right)_0$ (so that $m=2$).
This is spanned by the vectors
$v_{\alpha}:=
 e_{\alpha_1}\wedge e_{\alpha_2}\wedge
 e_{\alpha_3}\wedge e_{\alpha_4}$
(with dual basis the Pl\"ucker
coordinates $p_{\alpha}=v_{\alpha}^*$),
where
$\alpha\in\{\bar4\bar334,\bar4\bar224,\bar4\bar114,
            \bar3\bar223,\bar3\bar113,\bar2\bar112\}$.
The Northeast sequences are $\bar4\bar334$ and
$\bar4\bar224$.
The kernel of
$\Omega\contr\bullet:
 \left(\bigwedge^4\C^8\right)_0\rightarrow
 \left(\bigwedge^2\C^8\right)_0$
is spanned by the vectors
\[
K_{\bar4\bar224}=
v_{\bar4\bar224}-v_{\bar4\bar114}-v_{\bar3\bar223}+v_{\bar3\bar113}
\]
and
\[
K_{\bar4\bar334}=
v_{\bar4\bar334}-v_{\bar4\bar114}-v_{\bar3\bar223}+v_{\bar2\bar112}.
\]
To see this concretely, we compute:
\begin{eqnarray*}
\Omega\contr K_{\bar4\bar224}
&=&
 v_{\bar44}+v_{\bar22}-v_{\bar44}-v_{\bar11}
-v_{\bar33}-v_{\bar22}+v_{\bar33}+v_{\bar11}\\
&=&0,
\end{eqnarray*}
and similarly $\Omega\contr K_{\bar4\bar334}=0$.
The fibers of the map
$
\pi_4:\tangbin{4}{4}\rightarrow{\mathcal D}_4
$
are
\[
\pi_4^{-1}(\bar4\bar312,\bar2\bar134)=
\{\bar4\bar334,\bar2\bar112\}
\]
and
\[
\pi_4^{-1}(\bar4\bar213,\bar3\bar124)=
\{\bar4\bar224,\bar4\bar114,\bar3\bar223,\bar3\bar113\}.
\]
The expression for $p_{\bar4\bar114}$ as a linear combination
of Pl\"ucker coordinates indexed by Northeast sequences is
\[
\ell_{\bar4\bar114}=p_{\bar4\bar114}+
c_{\bar4\bar224}p_{\bar4\bar224}+
c_{\bar4\bar334}p_{\bar4\bar334},
\]
for some
$c_{\bar4\bar224},c_{\bar4\bar334}\in\C$,
which we can compute as follows.  Since
$0=\ell_{\bar4\bar114}(K_{\bar4\bar224})=
c_{\bar4\bar224}-1$, we have $c_{\bar4\bar224}=1$.
Similarly,
$0=\ell_{\bar4\bar114}(K_{\bar4\bar334})=
c_{\bar4\bar334}-1$, so $c_{\bar4\bar334}=1$.
Hence $\ell_{\bar4\bar114}=p_{\bar4\bar114}+
p_{\bar4\bar224}+p_{\bar4\bar334}$,
which agrees with~(\ref{lg4bin}).
\end{ex}

\subsection{Proof of the straightening law}

We find generators of
$(I_{d,n}{+}L_{d,n})\cap\C[{\mathcal D}_{d,n}]$
which express the quotient as an algebra with
straightening law on ${\mathcal D}_{d,n}$.
Such a generating set is automatically
a Gr\"obner basis with respect to the degree
reverse lexicographic term order where
variables are ordered by a refinement of the
doset order.  We begin with
a Gr\"obner basis $G_{I_{d,n}{+}L_{d,n}}$
for $I_{d,n}{+}L_{d,n}$
with respect to a similar term order.
For $\alpha^{(a)}\in\angbin{n}{n}_d$, write
$\check{\alpha}^{(a)}:=\alpha^{(a)}\vee(\alpha^t)^{(a)}$
and $\hat{\alpha}^{(a)}:=\alpha^{(a)}\wedge(\alpha^t)^{(a)}$,
so that
$\pi_n(\alpha^{(a)})=(\hat{\alpha}^{(a)},\check{\alpha}^{(a)})$.
We call an element $\alpha^{(a)}\in\angbin{n}{n}_d$
{\it Northeast} if $\alpha\in\angbin{n}{n}$
is Northeast.

Let $<$ be a linear refinement of the partial order
on ${\mathcal P}_{d,n}$ satisfying the
following conditions.  First, the Northeast
sequence is minimal among
those in a given fiber
of $\pi_n$.  This is possible
since every weight space is an antichain
({\it i.e.}, no two elements are comparable).
Second, $\alpha^{(a)}<\beta^{(b)}$ if
$(\hat{\alpha}^{(a)},\check{\alpha}^{(a)})$
is lexicographically smaller
than $(\hat{\beta}^{(b)},\check{\beta}^{(b)})$.

With respect to any such refinement, consider
the degree reverse lexicographic term order.
A reduced Gr\"{o}bner basis $G_{d,n}$ for
$I_{d,n}{+}L_{d,n}$ with respect to this
term order will have standard
monomials indexed by chains (in ${\mathcal P}_{d,n}$)
of Northeast partitions.  While every monomial supported
on a chain of Northeast partitions is standard modulo $I_{d,n}$,
this is not always the case modulo $I_{d,n}{+}L_{d,n}$.  In
other words, upon identifying each Northeast partition
appearing in a given monomial with an element of
${\mathcal D}_{d,n}$, we do not necessarily obtain a
monomial supported on a chain in ${\mathcal D}_{d,n}$.
It is thus necessary to identify precisely which Northeast
chains in $\angbin{n}{n}$ correspond to chains in
${\mathcal D}_{d,n}$ via the map $\pi_n$.

A monomial $p_{\alpha}^{(a)}p_{\beta}^{(b)}$
such that $\alpha^{(a)}<\beta^{(b)},(\beta^t)^{(b)}$
and $\alpha^{(a)}$, $\beta^{(b)}\in{\mathcal{NE}}$
cannot be reduced modulo $G_{I_{d,n}}$
or $G_{L_{d,n}}$.
On the other hand, if $\alpha^{(a)}<\beta^{(b)}$ (say),
but $\alpha^{(a)}$ and $(\beta^t)^{(b)}$
are incomparable (written
$\alpha^{(a)}\not\sim(\beta^t)^{(b)}$)
then there is a relation in 
$G_{I_{d,n}}$ with leading term
$p_{\alpha}^{(a)}p_{\beta^t}^{(b)}$.
It follows that the degree-two standard
monomials are indexed by Northeast
partitions $p_{\alpha}^{(a)}p_{\beta}^{(b)}$ with
$\alpha^{(a)}<\beta^{(b)},(\beta^t)^{(b)}$.

Conversely, any monomial
$p_{\alpha}^{(a)}p_{\beta}^{(b)}$ with
$\alpha^{(a)}<\beta^{(b)},(\beta^t)^{(b)}$
and $\alpha^{(a)}$, $\beta^{(b)}\in{\mathcal{NE}}$
cannot be the leading term of any
element of $G_{{I_{d,n}}{+}L_{d,n}}$.
To see this, observe that
$G_{{I_{d,n}}{+}L_{d,n}}$ is obtained
by Buchberger's algorithm~\cite{buc65} applied to
$G_{I_{d,n}}\cup G_{L_{d,n}}$, and
we may consider only
the $S$-polynomials
$S(f,g)$ with $f\in G_{I_{d,n}}$
and $g\in G_{L_{d,n}}$.  In this
case we may assume ${\rm in}_<g$ divides
${\rm in}_<f$.

Let $\alpha_0$ be the partition
such that ${\rm in}_<g=p_{\alpha_0}^{(a)}$
(that is, $g$ is the unique expression
of $p_{\alpha_0}^{(a)}$ as a linear
combination of Pl\"ucker coordinates
indexed by Northeast partitions),
and let $\alpha$ be the
unique Northeast partition such
that $\pi_n(\alpha_0)=\pi_n(\alpha)$.
By the reduced normal form given in
Theorem~\ref{normalform.thm},
$S(f,g)$ is the obtained by replacing
$p_{\alpha_0}^{(a)}$ with
$\pm p_{\alpha}^{(a)}+\ell$,
where $\ell$ is a linear combination
of Pl\"ucker coordinates
$p_{\gamma}^{(a)}$ with $\gamma$
Northeast and $\alpha_+<\gamma_+$.
This latter condition implies that
$\hat{\alpha}<\hat{\gamma}$
(also, $\check{\alpha}>\check{\gamma}$),
and therefore
$(\hat{\alpha},\check{\alpha})$
is lexicographically smaller
than $(\hat{\gamma},\check{\gamma})$.

Hence the standard monomials with
respect to the reduced Gr\"obner basis
$G_{{I_{d,n}}{+}L_{d,n}}$ are precisely
the monomials $p_{\alpha}^{(a)}p_{\beta}^{(b)}$ with
$\alpha^{(a)}<\beta^{(b)},(\beta^t)^{(b)}$
and $\alpha^{(a)}$, $\beta^{(b)}\in{\mathcal{NE}}$.

Recall that elements of the doset
${\mathcal D}_{d,n}$ are pairs
$(\alpha,\beta)$ of admissible
(Definition~\ref{invol.adm.def})
elements of $\angbin{n}{n}_d$ such
that (regarded as sequences):
\begin{itemize}
\item $\alpha<\beta$
\item $\alpha$ and $\beta$ have the same number
      of negative (or positive) elements
\end{itemize}
Equivalently, regarding $\alpha$ and
$\beta$ as partitions, the
elements of ${\mathcal D}_{d,n}$ are
pairs $(\alpha,\beta)$ of 
symmetric partitions such that
\begin{itemize}
\item $\alpha\subseteq\beta$,
\item $\alpha$ and $\beta$ have the same
Durfee square,
\end{itemize}
where the {\it Durfee square} of a
partition $\alpha$ is the largest square
subpartition $(p^p)\subseteq\alpha$ (for some
$p\leq n$).

\begin{thm}\label{asl.thm}
$\C[\angbin{n}{n}_d]/\langle I_{d,n}{+}L_{d,n}\rangle$ is an
algebra with straightening law on ${\mathcal D}_{d,n}$.
\end{thm}
\begin{proof}
Since standard monomials with respect to
a Gr\"{o}bner basis are linearly independent,
the arguments above
establish the first two conditions
in Definition~\ref{asl.def}.

To establish the third condition, note that it
suffices to consider the expression for a
degree-$2$ monomial as a sum of standard monomials.
For simplicity, we absorb the superscripts into
our notation and write $\alpha\in\angbin{n}{n}_d$
and similarly for the corresponding Pl\"ucker
coordinate.
Let
\begin{eqnarray}\label{redexpr.eqn}
p_{(\hat{\alpha},\check{\alpha})}p_{(\hat{\beta},\check{\beta})}
&=&
\sum_{j=1}^kc_j
 p_{(\hat{\alpha}_{j},\check{\alpha}_{j})}
 p_{(\hat{\beta}_{j},\check{\beta}_{j})}
\end{eqnarray}
be a reduced expression in $G_{I_{d,n}{+}L_{d,n}}$
for
$p_{(\hat{\alpha},\check{\alpha})}
p_{(\hat{\beta},\check{\beta})}$
as a sum of standard monomials.
That is,
$p_{(\hat{\alpha},\check{\alpha})}
p_{(\hat{\beta},\check{\beta})}$
is non-standard and
$p_{(\hat{\alpha}_j,\check{\alpha}_j)}
p_{(\hat{\beta}_j,\check{\beta}_j)}$
is standard for $j=1,\dotsc,k$.
We assume that
$\alpha$ (respectively, $\beta$)
be the unique Northeast
partition such that
$\pi_n(\alpha)=(\hat{\alpha},\check{\alpha})$
(respectively, $\pi_n(\beta)=(\hat{\beta},\check{\beta})$),
and similarly for each $\alpha_j$ and
$\beta_j$ appearing in~(\ref{redexpr.eqn}).

Fix $j=1,\dotsc,k$.  The standard monomial
$p_{(\hat{\alpha}_j,\check{\alpha}_j)}
p_{(\hat{\beta}_j,\check{\beta}_j)}$
is obtained by the reduction modulo $G_{L_{d,n}}$
of a standard monomial $p_{\gamma}p_{\delta}$
appearing in the straightening relation for
$p_{\alpha}p_{\beta}$, which is an element of the
Gr\"obner basis $G_{I_{d,n}}$.
If $\gamma$ and $\delta$ are both
Northeast, then nothing happens, {\it i.e.},
$\gamma=\alpha_j$ and $\delta=\beta_j$.
If $\gamma$ is not Northeast, then
we rewrite $p_{\gamma}$ as a linear combination
of Pl\"ucker coordinates indexed by Northeast
sequences.  Lemma~\ref{smaller.lemma}
ensures that the leading term of the new
expression is $p_{(\hat{\gamma},\check{\gamma})}$,
and the lower order terms
$p_{(\hat{\epsilon},\check{\epsilon})}$
satisfy $\hat{\epsilon}<\hat{\gamma}$.

It follows that the lexicographic
comparison in the third condition
of Definition~\ref{asl.def}
terminates with the first
Pl\"ucker coordinate.  That is, if
$(\hat{\alpha}_j\leq\check{\alpha}_j
\leq\hat{\beta}_j\leq\check{\beta}_j)$
is lexicographically smaller than
$(\hat{\alpha}\leq\check{\alpha}
\leq\hat{\beta}\leq\check{\beta})$,
then either $\hat{\alpha}_j<\hat{\alpha}$
or $\hat{\alpha}_j=\hat{\alpha}$
and $\hat{\alpha}_j<\hat{\alpha}$.
Therefore the reduction process
applied to $p_{\delta}$ does not
affect the result, and the third
condition is proven.

It remains to prove the fourth condition.
Suppose that
$(\hat{\alpha},\check{\alpha})$ and
$(\hat{\beta},\check{\beta})$
are incomparable elements of ${\mathcal D}_{d,n}$
($\alpha$ and $\beta$ Northeast).
This means that $\alpha$ is incomparable
to either $\beta$ or $\beta^t$
(possibly both).  Without loss
of generality, we will deal only
with the more complicated case that
$\alpha$ and $\beta^t$ are incomparable.
The hypothesis of the fourth
condition is that the set
$\{\hat{\alpha},\check{\alpha},\hat{\beta},\check{\beta}\}$
forms a chain in $\angbin{n}{n}_d$.
Up to interchanging the roles of
$\alpha$ and $\beta$, there are
two possible cases (see Figure~\ref{cases.fig}): either
$\hat{\alpha}<\hat{\beta}<\check{\alpha}<\check{\beta}$
or
$\hat{\alpha}<\hat{\beta}<\check{\beta}<\check{\alpha}$.
\begin{figure}[htb]
 \[
  \begin{picture}(300,130)(0,0)
   \put(40,20){\includegraphics[scale=0.7]{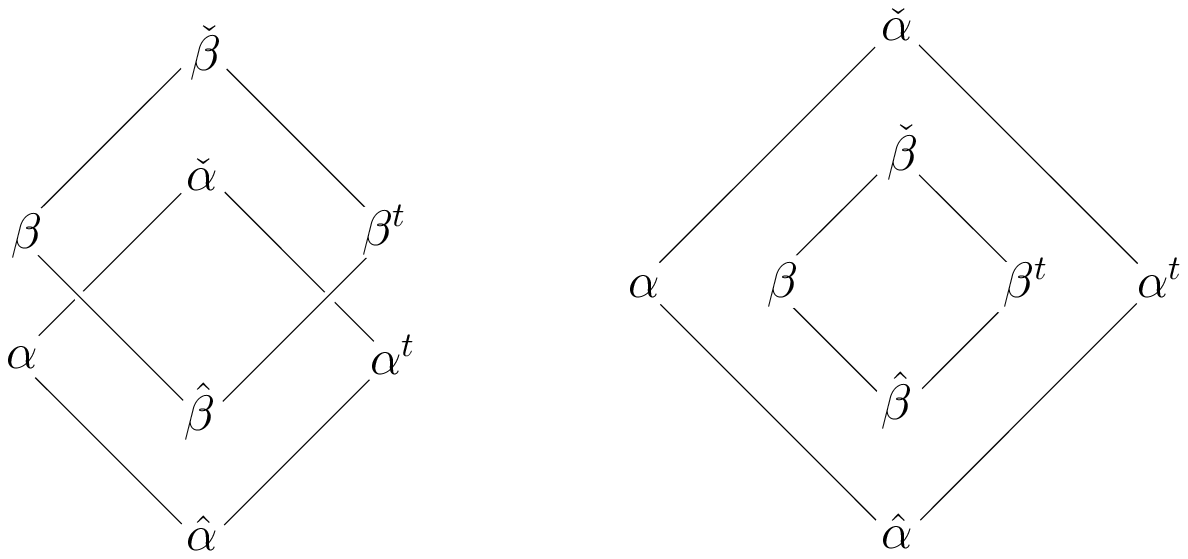}}
   \put(45,0)
    {$\hat{\alpha}<\hat{\beta}<\check{\alpha}<\check{\beta}$}
   \put(185,0)
    {$\check{\alpha}<\hat{\beta}<\check{\beta}<\check{\alpha}$}
  \end{picture}
 \]
 \caption{The two cases in the proof of fourth 
 condition of Definition~\ref{asl.def}.}
 \label{cases.fig}
\end{figure}

First, suppose
$\hat{\alpha}<\hat{\beta}<\check{\alpha}<\check{\beta}$.
Recall that for any
$\gamma_0\in\angbin{n}{n}_d$,
with Northeast sequence $\gamma$
in the same fiber of $\pi_n$,
the expression for the Pl\"ucker coordinate
$p_{\gamma_0}$ as a linear combination of
Pl\"ucker coordinates indexed by
Northeast sequences is supported
on Pl\"ucker coordinates $p_{\delta}$
such that $\delta_+\geq\gamma_+$
with equality if and only if
$\delta=\gamma$, and the Pl\"ucker
coordinate $p_{\gamma}$ appears
with coefficient $\pm1$
(Lemma~\ref{smaller.lemma}).

Upon replacing each Northeast
(or Southwest) partition with
its associated doset element using the
map $\pi_n$ from
Section~\ref{lag.ssec},
the first two terms of straightening relation
for $p_{\alpha}p_{\beta^t}$ are
\begin{eqnarray*}
p_{\alpha}p_{\beta^t}-
 p_{\alpha{\wedge}\beta^t}
  p_{\alpha\vee\beta^t}
&=&
p_{\alpha}p_{\beta^t}-
 \sigma p_{((\alpha\wedge\beta^t)^{\wedge},
 (\alpha\wedge\beta^t)^{\vee})}
  p_{((\alpha\vee\beta^t)^{\wedge},
  (\alpha\vee\beta^t)^{\vee})}+
  \mbox{lower order terms}\\
&=&
\sigma_{\beta}
 p_{(\hat{\alpha},\check{\alpha})}
 p_{(\hat{\beta},\check{\beta})}-
 \sigma
  p_{(\hat{\alpha},\hat{\beta})}
  p_{(\check{\alpha},\check{\beta})}
+\mbox{ lower order terms}\,,
\end{eqnarray*}
where $\sigma=\pm1$.
The second equation is justified as follows.
For any element $\alpha\in\angbin{n}{n}_d$,
recall that $\alpha_+$ (respectively,
$\alpha_-$) denotes the subsequence of
positive (negative) elements of $\alpha$.
This was previous defined for elements of
$\angbin{n}{n}$, but extends to elements
of $\angbin{n}{n}_d$ in the obvious way,
that is, by ignoring the superscript.
The condition
$\hat{\alpha}<\hat{\beta}<\check{\alpha}<\check{\beta}$
is equivalent to
$\bar\alpha_-^c<\bar\beta_-^c<\alpha_+<\beta_+$.
Note that this implies that
$\alpha\wedge\beta^t=\alpha_-\cup\bar\beta_-^c$
and
$\alpha\vee\beta^t=\bar\alpha_+^c\cup\beta_+$.
We compute in the distributive
lattice $\angbin{n}{n}_d$.
\begin{eqnarray*}
(\alpha\wedge\beta^t)\wedge(\alpha^t\wedge\beta)
&=&(\alpha_-\cup\bar\beta_-^c)
\wedge(\beta_-\cup\bar\alpha_-^c)\\
&=&\hat\alpha\\[5pt]
(\alpha\wedge\beta^t)\vee(\alpha^t\wedge\beta)
&=&
(\alpha_-\cup\bar\beta_-^c)
\vee
(\beta_-\cup\bar\alpha_-^c)\\
&=&\hat{\beta}
\end{eqnarray*}
Similarly,
$(\alpha\vee\beta^t)\wedge
(\alpha^t\vee\beta)=\check{\alpha}$
and
$(\alpha\vee\beta^t)\vee
(\alpha^t\vee\beta)=\check{\beta}$.

In the remaining case, we have
$\check{\alpha}<\hat{\beta}<\check{\beta}<\check{\alpha}$,
and it follows that
$\alpha\not\sim\beta$ and 
$\alpha\not\sim\beta^t$ both hold.
We use the relation for the incomparable pair
$\alpha\not\sim\beta^t$.
\begin{eqnarray*}
p_{\alpha}p_{\beta^t}-
 p_{\alpha{\wedge}\beta^t}
  p_{\alpha\vee\beta^t}
&=&
p_{\alpha}p_{\beta^t}-
 \sigma p_{((\alpha\wedge\beta^t)^{\wedge},
 (\alpha\wedge\beta^t)^{\vee})}
  p_{((\alpha\vee\beta^t)^{\wedge},
  (\alpha\vee\beta^t)^{\vee})}+
  \mbox{lower order terms}\\
&=&
\sigma_{\beta}
 p_{(\hat{\alpha},\check{\alpha})}
 p_{(\hat{\beta},\check{\beta})}-
 \sigma
  p_{(\hat{\alpha},\hat{\beta})}
  p_{(\check{\beta},\check{\alpha})}
+\mbox{ lower order terms}\,,
\end{eqnarray*}
where the second equality holds by
a similar computation in $\angbin{n}{n}_d$.
\end{proof}

The next result shows that the algebra with straightening law just
constructed is indeed the coordinate ring of $\LQ_d(n)$.
\begin{thm}\label{deadon.thm}
$\C[\angbin{n}{n}_{d}]/
 \langle I_{d,n}{+}L_{d,n}\rangle\cong\C[\LQ_d(n)]$.
\end{thm}
\begin{proof}
Let $I':=I(\LQ_d(n))$.  By definition, we have
$I_{d,n}{+}L_{d,n}\subseteq I'$.  Since
the degree and codimension of these ideals are equal, $I'$ is
nilpotent modulo $I_{d,n}{+}L_{d,n}$.  On the
other hand $I_{d,n}{+}L_{d,n}$ is radical,
so $I_{d,n}{+}L_{d,n}=I'$.
\end{proof}
\begin{cor}
The coordinate ring of any Schubert subvariety of
$L\Q_d(n)$ is an algebra with straightening law
on a doset, hence Cohen-Macaulay and Koszul. 
\end{cor}
\begin{proof}
For $\alpha^{(a)}\in{\mathcal D}_{d,n}$,
the Schubert variety $X_{\alpha^{(a)}}$
is defined by the vanishing of the
Pl\"ucker coordinates $p_{(\beta,\gamma)}^{(b)}$
for $\gamma^{(b)}\not\leq\alpha^{(a)}$.
The four conditions in Definition~\ref{asl.def}
are stable upon setting these variables to
zero, so we obtain an algebra with straightening
law on the doset
$\{(\beta,\gamma)^{(b)}\in{\mathcal D}_{d,n}
 \mid\gamma^{(b)}\leq\alpha^{(a)}\}$.

Let ${\mathcal D}\subseteq{\mathcal P}\times{\mathcal P}$
be a doset on the poset ${\mathcal P}$, $A$ any
algebra with straightening law on ${\mathcal D}$,
and $\C\{{\mathcal P}\}$
the unique {\it discrete}
algebra with straightening law on ${\mathcal P}$.
That is $\C\{{\mathcal P}\}$ has algebra generators
corresponding to the elements of ${\mathcal P}$,
and the straightening relations are $\alpha\beta=0$
if $\alpha$ and $\beta$ are incomparable elements of ${\mathcal P}$.
Then $A$ is Cohen-Macaulay if and only if $\C\{{\mathcal P}\}$
is Cohen-Macaulay~\cite{DeLa79}.

On the other hand, $\C\{{\mathcal P}\}$ is the face ring of the
order complex of ${\mathcal P}$.
The order complex of a locally upper semimodular poset
is shellable.  The face ring of a shellable simplicial
complex is Cohen-Macaulay~\cite{BH}.
By Proposition~\ref{distr.prop},
any interval in the poset ${\mathcal P}_{d,n}$
is a distributive lattice,
hence locally upper semimodular.
This proves that $\C[\LQ_d(n)]$
is Cohen-Macaulay.  The Koszul
property is a consequence of the
quadratic Gr\"obner basis consisting
of the straightening relations.
\end{proof}
The main results of this paper suggest that the
space of quasimaps is an adequate setting
for the study of the enumerative geometry of curves
into a general flag variety.  They also give
a new and interesting example of a family of varieties
whose coordinate rings are Hodge algebras.  We expect
that further study of the space of quasimaps into a
flag variety of general type will continue to yield
new results in these directions.
\bibliographystyle{amsplain}
\bibliography{all}

\end{document}